\documentclass{amsart}
\usepackage{paralist}
\usepackage{url}
\usepackage{mathrsfs}
\usepackage{color}
\usepackage{graphicx}
\usepackage{subfigure}
\usepackage{multirow}

\newcommand{\rs}{{\mathbb S}}
\newcommand{\re}{\mathbb{R}}
\newcommand{\mR}{\mathbb{R}}

\newcommand{\cpx}{\mathbb{C}}

\newcommand{\N}{\mathbb{N}}

\newcommand{\mF}{\mathcal{F}}
\newcommand{\mB}{\mathcal{B}}

\newcommand{\lmd}{\lambda}

\newcommand{\nn}{\nonumber}

\def\af{\alpha}

\def\rank{\mbox{rank}}

\def\mat{\mathbf{mat}}
\newcommand{\sig}{\sigma}
\newcommand{\Sig}{\Sigma}

\newcommand{\reff}[1]{(\ref{#1})}
\newcommand{\pt}{\partial}
\newcommand{\prm}{\prime}
\newcommand{\mc}[1]{\mathcal{#1}}

\newcommand{\mt}[1]{\mathtt{#1}}

\newcommand{\bdes}{\begin{description}}
\newcommand{\edes}{\end{description}}
\newcommand{\bal}{\begin{align}}
\newcommand{\eal}{\end{align}}
\newcommand{\bnum}{\begin{enumerate}}
\newcommand{\enum}{\end{enumerate}}
\newcommand{\bit}{\begin{itemize}}
\newcommand{\eit}{\end{itemize}}
\newcommand{\bea}{\begin{eqnarray}}
\newcommand{\eea}{\end{eqnarray}}
\newcommand{\be}{\begin{equation}}
\newcommand{\ee}{\end{equation}}
\newcommand{\baray}{\begin{array}}
\newcommand{\earay}{\end{array}}
\newcommand{\bsry}{\begin{subarray}}
\newcommand{\esry}{\end{subarray}}
\newcommand{\bca}{\begin{cases}}
\newcommand{\eca}{\end{cases}}
\newcommand{\bcen}{\begin{center}}
\newcommand{\ecen}{\end{center}}
\newcommand{\bbm}{\begin{bmatrix}}
\newcommand{\ebm}{\end{bmatrix}}
\newcommand{\bmx}{\begin{matrix}}
\newcommand{\emx}{\end{matrix}}
\newcommand{\bpm}{\begin{pmatrix}}
\newcommand{\epm}{\end{pmatrix}}
\newcommand{\btab}{\begin{tabular}}
\newcommand{\etab}{\end{tabular}}

\newtheorem{theorem}{Theorem}[section]

\theoremstyle{definition}

\newtheorem{exm}[theorem]{Example}
\newtheorem{alg}[theorem]{Algorithm}

\begin{document}

\title[Semidefinite Relaxations for Best Rank-$1$ Tensor Approximations]
{Semidefinite Relaxations for Best Rank-1 Tensor Approximations}

\author{Jiawang Nie}
\address{
Department of Mathematics, University of California, San Diego, 9500
Gilman Drive, La Jolla, CA 92093.} \email{njw@math.ucsd.edu}

\author{Li Wang}
\address{Department of Mathematics,
University of California, 9500 Gilman Drive, La Jolla, CA 92093.}
\email{liw022@math.ucsd.edu}

\begin{abstract}
This paper studies the problem of finding best rank-$1$
approximations for both symmetric and nonsymmetric tensors. For
symmetric tensors, this is equivalent to optimizing homogeneous
polynomials over unit spheres; for nonsymmetric tensors, this is
equivalent to optimizing multi-quadratic forms over multi-spheres.
We propose semidefinite relaxations, based on sum of squares
representations, to solve these polynomial optimization problems.
Their special properties and structures are studied. In
applications, the resulting semidefinite programs are often large
scale. The recent Newton-CG augmented Lagrangian method by Zhao, Sun
and Toh \cite{ZST10} is suitable for solving these semidefinite
relaxations. Extensive numerical experiments are presented to show
that this approach is efficient in getting best rank-$1$
approximations.
\end{abstract}

\keywords{form, polynomial, relaxation, rank-$1$ approximation,
semidefinite program, sum of squares, tensor}

\subjclass{15A18, 15A69, 90C22}


\maketitle

\section{Introduction}

Let $m$ and $n_1,\ldots, n_m$ be positive integers. A tensor of
order $m$ and dimension $(n_1,\ldots,n_m)$ is an array $\mc{F}$ that
is indexed by integer tuples $(i_1, \ldots, i_m)$ with $1 \leq i_j
\leq n_j$ ($j=1,\ldots, m$), i.e.,
\[
\mc{F} = (\mc{F}_{i_1,\ldots, i_m})_{ 1\leq i_1 \leq n_1, \ldots, 1
\leq i_m \leq n_m  }.
\]
The space of all such tensors with real (resp., complex) entries is
denoted as $\re^{n_1 \times \cdots \times n_m}$ (resp., $\cpx^{n_1
\times \cdots \times n_m}$). Tensors of order $m$ are called
$m$-tensors. When $m$ equals $1$ or $2$, they are regular vectors or
matrices. When $m=3$ (resp., $4$), they are called cubic (resp.,
quartic) tensors. A tensor $\mc{F} \in \re^{n_1 \times \cdots \times
n_m}$ is symmetric if $n_1 = \cdots = n_m$ and
\[
\mc{F}_{i_1,\ldots, i_m} = \mc{F}_{j_1,\ldots, j_m} \quad \mbox{ for
all } \quad (i_1,\ldots, i_m) \sim (j_1,\ldots, j_m),
\]
where $\sim$ means that $(i_1,\ldots, i_m)$ is a permutation of
$(j_1,\ldots, j_m)$. We define the norm of a tensor $\mF$ as: \be
\label{tensor:norm} \|\mF\|  = \left( \sum_{i_1=1}^{n_1}\cdots
\sum_{i_m=1}^{n_m} |\mF_{i_1, \ldots,i_m}|^2  \right)^{1/2}. \ee For
$m=1$, $\|\mF\|$ is the vector 2-norm, and for $m=2$, $\|\mF\|$ is
the matrix Frobenius norm.

Every tensor can be expressed as a linear combination of outer
products of vectors. For vectors $u^1 \in \cpx^{n_1}, \ldots, u^m
\in \cpx^{n_m}$, their outer product $u^1 \otimes \cdots \otimes
u^m$ is the tensor in $\cpx^{n_1 \times \cdots \times n_m}$ such
that for all $1 \leq i_j \leq n_j$ ($j=1,\ldots, m$)
\[
(u^1 \otimes \cdots \otimes u^m)_{i_1,\ldots,i_m} = (u^1)_{i_1}
\cdots (u^m)_{i_m}.
\]
For every tensor $\mc{F}$ of order $m$, there exist tuples
$(u^{i,1},\ldots, u^{i,m})$ ($i=1,\ldots,r$), with each $u^{i,j} \in
\cpx^{n_j}$, such that \be \label{dcmp:rank-r} \mF
=\sum\limits_{i=1}^r u^{i,1} \otimes \cdots \otimes u^{i,m}. \ee The
smallest $r$ in the above is called the rank of $\mF$, and is
denoted as $\rank \, \mF$. When $\rank \,\mF=r$, \reff{dcmp:rank-r}
is called a rank decomposition of $\mF$, and we say that $\mF$ is a
rank-$r$ tensor.

Tensor problems have wide applications in chemometrics, signal
processing and high order statistics \cite{Com00}. For the theory
and applications of tensors, we refer to Comon et al.~\cite{CGLM08},
Kolda and Bader \cite{KolBad09}, and Landsberg \cite{Land12}. When
$m \geq 3$, determining ranks of tensors and computing rank
decompositions are NP-hard (cf.~Hillar and Lim \cite{HiLi13}). We
refer to Brachat et al.~\cite{BCMT10}, Bernardi et
al.~\cite{BBCM13}, Oeding and Ottaviani~\cite{OedOtt13} for tensor
decompositions. When $r>1$, the problem of finding best rank-$r$
approximations may be ill-posed, because a best rank-$r$
approximation might not exist, as discovered by De Silva and Lim
\cite{DeSLim08}. However, a best rank-$1$ approximation always
exists. It is also NP-hard to compute best rank-$1$ approximations.
This paper studies best real rank-$1$ approximations for real
tensors (i.e., tensor entries are real numbers). We begin with some
reviews on this subject.

\subsection{Nonsymmetric tensor approximations}

Given a tensor $\mF \in \re^{n_1 \times \cdots \times n_m}$, we say
that a tensor $\mB$ is a best rank-$1$ approximation of $\mF$ if it
is a minimizer of the least squares problem \be
\label{bstaprox:nosym} \min_{ \mc{X} \in \re^{n_1 \times \cdots
\times n_m}, \rank \,\mc{X} =1} \qquad \|\mF - \mc{X} \|^2. \ee This
is equivalent to a homogeneous polynomial optimization problem, as
shown in \cite{LMV00b}.
For convenience of description, we define the homogeneous polynomial
\be \label{df:F(x):nsy} F ( x^{1}, \ldots, x^{m}) := \sum_{1 \leq
i_1 \leq n_1, \ldots, 1 \leq i_m \leq n_m} \mF_{i_{1},\ldots, i_{m}}
(x^{1})_{i_1} \cdots (x^{m})_{i_m}, \ee which is in $x^{1} \in
\re^{n_1}, \ldots, x^{m} \in \re^{n_m}$. Note that $F( x^{1},
\ldots, x^{m}) $ is a multi-linear form (a form is a homogeneous
polynomial), since it is linear in each $x^{j}$. De Lathauwer, De
Moor and Vandewalle \cite{LMV00b} proved the following result.

\begin{theorem} \emph{(}\cite[Theorem 3.1]{LMV00b}\emph{)} \label{thm:LLM00}
For a tensor $\mF \in \re^{n_1 \times \cdots \times n_m}$, the
rank-$1$ approximation problem \reff{bstaprox:nosym} is equivalent
to the maximization problem \be \label{maxF:sph:nsym} \left\{
\baray{cl} \underset{ x^{1} \in \re^{n_1}, \ldots, x^{m} \in
\re^{n_m} }{\max} &
          |F ( x^{1}, \ldots, x^{m})| \\
 \emph{s.t.} & \|x^{1}\| = \cdots =  \|x^{m}\|=1, \earay
 \right.
 \ee
that is, a rank-$1$ tensor $\lmd \cdot (u^{1} \otimes \cdots \otimes
u^{m})$, with $\lmd \in \re$ and each $\|u^{i}\| =1$, is a best
rank-$1$ approximation for $\mF$ if and only if $(u^{1}, \ldots,
u^{m})$ is a global maximizer of \reff{maxF:sph:nsym} and $\lmd =
F(u^{1}, \ldots, u^{m})$. Moreover, it also holds that
\[
\| \mF - \lmd \cdot (u^{1} \otimes \cdots \otimes u^{m}) \|^2 = \|
\mF \|^2 - \lmd^2.
\]
\end{theorem}

By Theorem~\ref{thm:LLM00}, to find a best rank-$1$ approximation,
it is enough to find a global maximizer of the multi-linear
optimization problem \reff{maxF:sph:nsym}. There exist methods on
finding rank-$1$ approximation, like the alternating least squares
(ALS), truncated high order singular value decomposition (T-HOSVD),
higher-order power method (HOPM), and Quasi-Newton methods. We refer
to Zhang and Golub~\cite{Tongrank2001}, De Lathauwer, De Moor and
Vandewalle \cite{HOSVD2000,LMV00b}, Savas and Lim \cite{SavLim10}
and the references therein. An advantage of these methods is that
they can be easily implemented. These kinds of methods typically
generate a sequence that converges to a locally optimal rank-$1$
approximation or even just a stationary point. Even for the lucky
cases that they get globally optimal rank-$1$ approximations, it is
usually very difficult to verify the global optimality by these
methods.

\subsection{Symmetric tensor approximations}

Let $\mt{S}^{m}(\re^{n})$ be the space of real symmetric tensors of
order $m$ and in dimension $n$. Given $\mF \in \mt{S}^m(\re^n)$, we
say that $\mB$ is a best rank-$1$ approximation of $\mF$ if it is a
minimizer of the optimization problem \be \label{mindist:sym}
\min_{\mc{X} \in \re^{n\times \cdots \times n}, \, \rank \,\mc{X} =1
} \qquad \|\mF-\mc{X} \|^2. \ee When $\mF$ is symmetric, Zhang, Ling
and Qi \cite{ZLQ12} showed that \reff{mindist:sym} has a global
minimizer that belongs to $\mt{S}^m(\re^n)$, i.e.,
\reff{mindist:sym} has an optimizer that is a symmetric tensor. It
is possible that a best rank-$1$ approximation of a symmetric tensor
might not be symmetric. But there is always at least one global
minimizer of \reff{mindist:sym} that is a symmetric rank-1 tensor.
Therefore, for convenience, by best rank-$1$ approximation for
symmetric tensors, we mean best symmetric rank-$1$ approximation.

A symmetric tensor in $\mt{S}^m(\re^n)$ is rank-$1$ if and only if
it equals $\lambda \cdot (u \otimes \cdots \otimes u)$ for some
$\lmd \in \re$ and $u\in \re^n$. For convenience, denote $u^{\otimes
m} :=u \otimes \cdots \otimes u$ ($u$ is repeated $m$ times). In the
spirit of Theorem~\ref{thm:LLM00} and the work~\cite{ZLQ12},
\reff{mindist:sym} is equivalent to the optimization problem \be
\label{max|f(x)|:sym}   \max_{x \in \re^n } \quad |f(x)| \quad
\text{s.t.} \quad \|x \| =  1, \ee where $f(x)  := F(x, \cdots ,x)$.
Therefore, if $u$ is a global maximizer of \reff{max|f(x)|:sym} and
$\lmd = f(u)$, then $\lambda \cdot u^{\otimes m}$ is a best rank-$1$
approximation of $\mF$. Clearly, to solve \reff{max|f(x)|:sym}, we
need to solve two maximization problems: \be  \label{sym:2:maxf}
\text{(I)}\quad \max\limits_{ x \in \mathbb{S}^{n-1}} f(x), \qquad
\text{(II)}\quad \max\limits_{ x \in \mathbb{S}^{n-1}} -f(x), \ee
where $\mathbb{S}^{n-1} :=\{x\in \mR^n:\|x\|=1\} $ is the $n-1$
dimensional unit sphere. Suppose $u^+,u^-$ are global maximizers of
(I), (II) in \reff{sym:2:maxf} respectively. By
Theorem~\ref{thm:LLM00}, if $|f(u^+)| \geq |f(u^-)|$, $f(u^+) \cdot
(u^+)^{\otimes m}$ is the best rank-$1$ approximation; otherwise,
$f(u^-) \cdot (u^-)^{\otimes m}$ is the best.

For an introduction to symmetric tensors, we refer to Comon, Golub,
Lim and Mourrain \cite{CGLM08}. Finding best rank-$1$
approximations for symmetric tensors is also NP-hard when $m\geq 3$.
There exist methods for computing rank-$1$ approximations for
symmetric tensors. When HOPM is directly applied, it is often
unreliable for attaining a good symmetric rank-$1$ approximation, as
pointed out in \cite{LMV00b}. To get good symmetric rank-$1$
approximations, Kofidis and Regalia \cite{SHOPM2002} proposed a
symmetric higher-order power method (SHOPM), Zhang, Ling and Qi
\cite{ZLQ12} proposed a modified power method. These methods can be
easily implemented. Like for nonsymmetric tensors, they often
generate a locally optimal rank-$1$ approximation or even just a
stationary point. Even for the lucky cases that a globally optimal
rank-$1$ approximation is found, these methods typically have
difficulty to verify its global optimality.
%
%
The problem \reff{max|f(x)|:sym} is related to extreme
$Z$-eigenvalues of symmetric tensors. Recently, Hu, Huang and
Qi~\cite{Hu2013} proposed a method for computing extreme
$Z$-eigenvalues for symmetric tensors of even orders. It is to
solve a sequence of semidefinite relaxations based on sum of squares
representations.

\subsection{Contributions of the paper.}

In this paper, we propose a new approach for computing best rank-$1$
tensor approximations, i.e., using semidefinite program (SDP)
relaxations. As we have seen, for nonsymmetric tensors, the problem
is equivalent to optimizing a multi-linear form over multi-spheres,
i.e., \reff{maxF:sph:nsym}; for symmetric tensors, it is equivalent
to optimizing a homogeneous polynomial over the unit sphere, i.e.,
\reff{max|f(x)|:sym}. Recently, there is extensive work on solving
polynomial optimization problems by using semidefinite relaxations
and sum of squares representations, e.g., Lasserre~\cite{Las01},
Parrilo and Sturmfels~\cite{ParStu03}, Nie and
Wang~\cite{NieLarge2012}. These relaxations are often tight, and
generally they are able to get global optimizers, which can be
verified mathematically. We refer to Lasserre's book \cite{Las09}
and Laurent's survey \cite{ML2009} for an overview for the work in
this area.

For a nonsymmetric tensor $\mF$, to get a best rank-$1$
approximation is equivalent to solving the multi-linear optimization
problem \reff{maxF:sph:nsym}. When $\mF$ is symmetric, to get a best
rank-$1$ approximation for $\mF$, we need to solve the homogeneous
optimization problem \reff{max|f(x)|:sym}. We solve these polynomial
optimization problems by using semidefinite relaxations, based on
sum of squares representations. In applications, the resulting
semidefinite programs are often large scale. The traditional
interior-point methods for semidefinite programs are generally too
expensive for solving them. 
The recent Newton-CG augmented Lagrangian method by Zhao, Sun and
Toh \cite{ZST10} is very efficient for solving such big semidefinite
programs, as shown in \cite{NieLarge2012}. In the paper, we use this
method to compute best rank-1 approximations for tensors.

The paper is organized as follows. In section 2, we show how to find
best rank-$1$ approximations by using semidefinite relaxations and
propose numerical algorithms based on them. Their special structures
and properties are also studied. In Section~3, we present extensive
numerical experiments to show the efficiency of these semidefinite
relaxations. In Section~4, we make some discussions about this
approach and future work.

\vspace{8pt} \noindent {\bf Notation} \, The symbol $\N$ (resp.,
$\re$, $\mathbb{C}$) denotes the set of nonnegative integers (resp.,
real numbers, complex numbers). For two tensors $\mc{X},\mc{Y} \in
\re^{n_1 \times \cdots n_m}$, define their inner product as
\[
\langle \mc{X}, \mc{Y} \rangle  := \sum_{1 \leq i_1 \leq n_1,\ldots,
1\leq i_m \leq n_m } \mc{X}_{i_1,\ldots,i_m}
\mc{Y}_{i_1,\ldots,i_m}.
\]
For $t\in \re$, $\lceil t\rceil$ denotes the smallest integer that
is not smaller than $t$. For integer $n>0$, $[n]$ denotes the set
$\{1,\cdots,n\}$. For $\alpha =(\af_1,\ldots, \af_n) \in \N^n$,
$|\alpha|:=\alpha_1+\cdots+\alpha_n$. Denote $\N^n_m = \{ \af \in
\N^n:\, |\af| = m\}$.
For $x =(x_1,\ldots,x_n) \in \re^n$, $x^\af$ denotes
$x_1^{\af_1}\cdots x_n^{\af_n}$. Let $\re[x]$ be the ring of
polynomials with real coefficients in $x$. A polynomial $p$ is said
to be sum of squares (SOS) if $p= p_1^2 + \cdots + p_k^2$ for some
$p_1,\ldots,p_k \in \re[x]$. The symbol $\Sig_{n,m}$ denotes the
cone of SOS forms in $(x_1,\ldots,x_n)$ and of degree $m$. For a set
$T$, $|T|$ denotes its cardinality. The symbol $e_i$ denotes the
$i$-th unit vector, i.e., $e_i$ is the vector whose $i$-th entry
equals one and all others equal zero. For a matrix $A$, $A^T$
denotes its transpose. For a symmetric matrix $W$, $W\succeq
0$(resp., $\succ 0$) means that $W$ is positive semidefinite (resp.,
definite). For any vector $u\in \re^N$, $\| u \| = \sqrt{u^Tu}$
denotes the standard Euclidean 2-norm.
%
%

\section{Semidefinite Relaxations and Algorithms}
\label{section:sdp:relaxation} \setcounter{equation}{0}

To find best rank-$1$ tensor approximations is equivalent to solving
some homogeneous polynomial optimization problems with sphere
constraints. In this section, we show how to solve them by using
semidefinite relaxations based on sum of squares representations,
and study their properties.

\subsection{Symmetric tensors of even orders}
\label{sec:hmgopt} \setcounter{equation}{0}

Let $\mF \in \mt{S}^m(\re^n)$ with $m=2d$ even. To get a best
rank-$1$ approximation of $\mF$, we need to solve
\reff{max|f(x)|:sym}, i.e., to maximize $|f(x)|$ over the unit
sphere. For this purpose, we need to find the maximum and minimum of
$f(x)$ over $\mathbb{S}^{n-1}$.

First, we consider the maximization problem: \be
\label{maxf:||x||=1} f_{\max}:= \max \quad f(x) \quad \text{s.t.}
\quad x^Tx = 1. \ee The form $f$ is in $x:=(x_1,\ldots,x_n)$,
determined by the tensor $\mF$ as \be \label{hmf:sym} f(x) = \sum_{1
\leq i_1, \ldots, i_m \leq n} \mF_{i_1,\ldots, i_m} x_{i_1}\cdots
x_{i_m}. \ee Let $[x^d]$ be the monomial vector:
\[
[x^d] := \bbm x_1^d & x_1^{d-1}x_2 & \cdots & x_{1}^{d-1}x_n &
\cdots & x_n^d \ebm^T.
\]
Its length is $\binom{n+d-1}{d}$. The outer product $[x^d][x^d]^T$
is a symmetric matrix with entries being monomials of degree $m$.
Let $A_{\alpha}$ be symmetric matrices such that
\[
[x^d][x^d]^T = \sum_{ \af \in \N_m^n }   A_{\alpha} x^\af.
\]
For $y\in \re^{ \N^n_m }$, define the matrix-valued function $M(y)$
as
\[
M(y) :=  \sum_{ \af \in \N^n_m }   A_{\alpha} y_\af.
\]
Then $M(y)$ is a linear pencil in $y$ (i.e., $M(y)$ is a linear
matrix-valued function in $y$). Let $f_\af, g_\af$ be the
coefficients such that
\[
f  :=  \sum_{ \af \in \N^n_m }   f_{\af} x^\af, \quad g :=(x^Tx)^d =
\sum_{ \af \in \N^n_m }   g_{\af} x^\af.
\]
For $y\in \re^{ \N^n_m }$, define
\[
\langle f, y \rangle : =  \sum_{ \af \in \N^n_m }   f_{\af} y_\af,
\quad \langle g, y \rangle : =    \sum_{ \af \in \N^n_m }   g_{\af}
y_\af.
\]

A semidefinite relaxation of \reff{maxf:||x||=1} is (cf.
\cite{Nieform2012,NieLarge2012}) \be \label{maxf:sdpr}
f_{\max}^{\text{sdp}} := \underset{y \in \re^{\N^n_m} }{\max}
 \quad  \langle f, y \rangle  \quad
 \text{s.t.} \quad M(y) \succeq 0, \quad \langle g, y \rangle = 1.
\ee It can be shown that the the dual of the above is \be
\label{maxf:sosr} \min \quad \gamma \quad \text{s.t.} \quad \gamma g
- f \in \Sig_{n,m}. \ee In the above, $\Sig_{n,m}$ denotes the cone
of SOS forms of degree $m$ and in variables $x_1,\ldots, x_n$.
Clearly, $f_{\max}^{\text{sdp}} \geq f_{\max}$. When
$f_{\max}^{\text{sdp}} = f_{\max}$, we say that the semidefinite
relaxation \reff{maxf:sdpr} is tight.

The dual \reff{maxf:sosr} is a linear optimization problem with SOS
type constraints. Recently, Hu, Huang and Qi \cite{Hu2013} proposed
an SOS relaxation method for computing extreme
Z-eigenvalues for symmetric tensors of even orders. Since not every
nonnegative form is SOS, they consider a sequence of nesting SOS
relaxations. The problem \reff{maxf:sosr} is equivalent to the
lowest order relaxation in \cite[\S4]{Hu2013}. In practice, the
relaxation \reff{maxf:sosr} is often tight, and it is frequently
used because of its simplicity. The SOS relaxation \reff{maxf:sosr}
was also proposed in \cite{Nieform2012} for optimizing forms
over unit spheres. Its approximation quality was also analyzed there.

The feasible set of \reff{maxf:sdpr} is compact, because
\[
\mbox{Trace}(M(y)) \leq \langle g, y\rangle = 1.
\]
So \reff{maxf:sdpr} always has a maximizer, say, $y^*$. If
$\rank\,M(y^*) =1$, then \reff{maxf:sdpr} is a tight relaxation. In
such case, there exists $v^+ \in \mathbb{S}^{n-1}$ such that $y^* =
[(v^+)^m]$, because of the structure of $M(y)$. Then, it holds that
\[
 f(v^+) = \langle f, y^* \rangle = f_{\max}^{\text{sdp}} \geq f_{\max}.
\]
This implies that $v^+$ is a global maximizer of
\reff{maxf:||x||=1}. The vector $v^+$ can be chosen numerically as
follows. Let $s \in [n]$ be the index such that
\[
y^*_{2de_s}  = \max_{ 1 \leq i \leq n} y^*_{2de_i}.
\]
Then choose $v^+$ as the normalization:
%
%
\be \label{cho:v+} \hat{u} = (y^*_{ (2d-1)e_{s} + e_1}, \ldots,
y^*_{(2d-1)e_{s} + e_n}),\quad  v^+ = \hat{u}/\|\hat{u}\|. \ee If
$\rank(M(y^*))>1$ but $M(y^*)$ satisfies a further rank condition,
then \reff{maxf:sdpr} is also tight
(cf.~\cite[Section~2.2]{NieLarge2012}). In computations, no matter
\reff{maxf:sdpr} is tight or not, the vector $v^+$ selected as in
\reff{cho:v+} can be used as an approximation for a maximizer of
\reff{maxf:||x||=1}.

Second, we consider the minimization problem: \be
\label{minf:||x||=1} f_{\min}:= \min \quad f(x) \quad \text{s.t.}
\quad x^Tx = 1. \ee Similarly, a semidefinite relaxation of
\reff{minf:||x||=1} is \be  \label{minf:sdpr} f^{\text{sdp}}_{\min}
:= \underset{z \in \re^{\N^n_m} }{\min}  \quad   \langle f, z
\rangle \quad \text{s.t.} \quad M(z) \succeq 0, \quad  \langle g, z
\rangle = 1. \ee Its dual optimization problem can be shown to be
\be \max \quad \eta \quad \text{s.t.} \quad
 f - \eta g   \in \Sig_{n,m}.
\ee Let $z^*$ be a minimizer of \reff{minf:sdpr}. Similarly, if
$\rank(M(z^*)) =1$, then \reff{minf:sdpr} is a tight relaxation. In
such case, a global minimizer $v^-$ can be found as follows: let $t
\in [n]$ be the index such that
\[
z^*_{2de_t}  = \max_{ 1 \leq i \leq n} z^*_{2de_i},
\]
then choose $v^-$ as the normalization:
\be \label{sel:v-} \tilde{u} =(z^*_{ (2d-1)e_{t} + e_1}, \ldots,
z^*_{(2d-1)e_{t} + e_n}), \quad   v^{-} = \tilde{u}/\| \tilde{u} \|.
\ee When $\rank\, M(z^*) >1$, (\ref{minf:sdpr}) might not be a tight
relaxation. In computations, the vector $v^-$ can be used as an
approximation for a minimizer of \reff{minf:||x||=1}.

Combining the above and inspired by Theorem~\ref{thm:LLM00}, we get
the following algorithm.

\begin{alg}  \label{alg:sym:even}
Rank-$1$ approximations for even symmetric tensors.
\\
\noindent {\bf Input:}\, A symmetric tensor $\mF \in
\mt{S}^m(\re^n)$ with an even order $m=2d$.

\noindent {\bf Output:}\,  A rank-$1$ tensor $\lmd \cdot u^{\otimes
m}$, with $\lmd \in \re$ and $u \in \mathbb{S}^{n-1}$.

\noindent {\bf Procedure:} \quad \bdes

\item [Step 1] Solve the semidefinite relaxation~\reff{maxf:sdpr}
and get an optimizer $y^*$.

\item [Step 2] Choose $v^+$ as in \reff{cho:v+}.
If $\rank\,M(y^*) = 1$, let $u^+=v^+$; otherwise, apply a nonlinear
optimization method to get a better solution $u^+$ of
\reff{maxf:||x||=1}, by using $v^+$ as a starting point. Let $\lmd^+
= f(u^+)$.

\item [Step 3] Solve the semidefinite relaxation~\reff{minf:sdpr}
and get an optimizer $z^*$.

\item [Step 4] Choose $v^{-}$ as in \reff{sel:v-}.
If $\rank\,M(z^*) = 1$, let $u^-=v^-$; otherwise, apply a nonlinear
optimization method to get a better solution $u^-$ of
\reff{minf:||x||=1}, by using $v^{-}$ as a starting point. Let
$\lmd^- = f(u^-)$.

\item [Step 5] If $|\lmd^+| \geq |\lmd^-|$,
output $(\lmd, u) := (\lmd^+, u^+)$; otherwise, output $(\lmd, u) :=
(\lmd^-, u^-)$.

\edes

\end{alg}

\noindent {\it Remark:} In Algorithm~\ref{alg:sym:even}, if
$\rank\,M(y^*) = \rank\,M(z^*) = 1$, then the output $\lmd \cdot
u^{\otimes m}$ is a best rank-$1$ approximation of $\mF$. If this
rank condition is not satisfied, then $\lmd \cdot u^{\otimes m}$
might not be best. The approximation qualities of semidefinite
relaxations \reff{maxf:sdpr} and \reff{minf:sdpr} are analyzed in
\cite[Section~2]{Nieform2012}.

When $m=2$ or $(n,m) = (3,4)$, the semidefinite relaxations
\reff{maxf:sdpr} and \reff{minf:sdpr} are always tight. This is
because every bivariate, or ternary quartic, nonnegative form is
always SOS (cf.~\cite{sosexm2000}). For other cases of $(n,m)$, this
is not necessarily true. However, this does not occur very often in
applications. In our numerical experiments, the semidefinite
relaxations \reff{maxf:sdpr} and \reff{minf:sdpr} are often tight.

We want to know when the ranks of $ M(y^*)$ and $M(z^*)$ are equal
to one. Clearly, when they have rank one, the relaxations
\reff{maxf:sdpr} and \reff{minf:sdpr} must be tight. Interestingly,
the reverse is often true, although it might be false sometimes (for
very few cases). This fact was observed in the field of polynomial
optimization. However, we are not able to find suitable references
for explaining this fact. Here, we give a natural geometric
interpretation for this phenomena, for lack of suitable references.
Let $\pt \Sig_{n,m}$ be the boundary of the cone $\Sig_{n,m}$.

\begin{theorem}  \label{thm:symeven:rk=1}
Let $f_{\max},f_{\max}^{\emph{sdp}}, f_{\min},f_{\min}^{\emph{sdp}},
y^*, z^*$ be as above. \bit

\item [(i)]  Suppose $f_{\max} = f_{\max}^{\emph{sdp}}$.
If $f_{\max} \cdot g - f$ is a smooth point of $\pt\Sig_{n,m}$, then
$\rank \, M(y^*) = 1$.

\item [(ii)] Suppose $f_{\min} = f_{\min}^{\emph{sdp}}$.
If $f - f_{\min} \cdot g$ is a smooth point of $\pt\Sig_{n,m}$, then
$\rank \, M(z^*) = 1$.

\eit
\end{theorem}

\begin{proof}
(i) Let $\mu$ be the uniform probability measure on the unit sphere
$\mathbb{S}^{n-1}$, and let $y^{\mu} := \int [x^m] d\mu \in
\re^{\N_m^n}$. Then, we can show that $M(y^{\mu}) \succ 0$ and
$\langle g, y^{\mu} \rangle = 1$. This shows that $y^{\mu}$ is an
interior point of \reff{maxf:sdpr}. So, the strong duality holds,
that is, the optimal values of \reff{maxf:sdpr} and \reff{maxf:sosr}
are equal, and \reff{maxf:sosr} achieves the optimal value
$f_{\max}^{\text{sdp}}$. The form $\sig: = f_{\max} \cdot g - f$
belongs to $\Sig_{n,m}$, because $f_{\max} = f_{\max}^{\text{sdp}}$.
Let $u$ be a maximizer of $f$ on $\mathbb{S}^{n-1}$. Then,
$\sig(u)=0$. So, $\sig$ lies on the boundary $\pt\Sig_{n,m}$. Let
$\hat{y} = [u^m]$. Denote by $\re[x]_m$ the space of all forms in
$x$ and of degree $m$. Then
\[
\mc{H}_{\hat{y}} := \{ p \in \re[x]_m: \, \langle p, \hat{y} \rangle
= 0 \}
\]
is a supporting hyperplane of $\Sig_{n,m}$ through $\sig$, because
\[
\langle p, \hat{y} \rangle = p(u) \geq 0, \quad \forall \, p \in
\Sig_{n,m}
\]
and $\langle \sig, \hat{y} \rangle = \sig(u) = 0$. Because
$M(y^*)\succeq 0$ and $\langle p, y^* \rangle \geq 0$ for all $p \in
\Sig_{n,m}$, the hyperplane
\[
\mc{H}_{y^*} := \{ p \in \re[x]_m: \, \langle p, y^* \rangle = 0 \}
\]
also supports $\Sig_{n,m}$ through $\sig$. Since $\sig$ is a smooth
point of $\pt\Sig_{n,m}$, there is a unique supporting hyperplane
$\mc{H}$ through $\sig$. So, $y^* = \hat{y} = [u^m]$. This implies
that $M(y^*) = M( [ u^m ] ) = [u^d ] [u^d]^T$ has rank one, where
$d=m/2$.

(ii) The proof is the same as for (i).
\end{proof}

By Theorem~\ref{thm:symeven:rk=1}, when semidefinite relaxations
\reff{maxf:sdpr} and \reff{minf:sdpr} are tight, we often have
$\rank\, M(y^*)=\rank\,M(z^*)=1$. This fact was observed in our
numerical experiments. The reason is that the set of nonsmooth
points of the boundary $\pt\Sig_{n,m}$ has a strictly smaller
dimension than $\pt\Sig_{n,m}$.

\subsection{Symmetric tensors of odd orders}
\label{section:odd:opt}

Let $\mF \in \mt{S}^m(\re^n)$ with $m=2d-1$ odd. To get a best
rank-$1$ approximation of $\mF$, we need to solve the optimization
problem \reff{max|f(x)|:sym}. Since the form $f$, defined as in
\reff{hmf:sym}, has the odd degree $m$, \reff{max|f(x)|:sym} is
equivalent to \reff{maxf:||x||=1}. We can not directly apply a
semidefinite relaxation to solve \reff{maxf:||x||=1}. For this
purpose, we use a trick that is introduced in
\cite[Section~4.2]{Nieform2012}.

Let $f_{\max}, f_{\min}$ be as in \reff{maxf:||x||=1},
\reff{minf:||x||=1} respectively. Since $f$ is an odd form,
\[
f_{\max} = -f_{\min} \geq 0.
\]
Let $x_{n+1}$ be a new variable, in addition to
$x=(x_1,\ldots,x_n)$. Let
\[
\tilde{x} := (x_1,\ldots,x_n, x_{n+1}), \quad \tilde{f}(\tilde{x})
:= f(x) x_{n+1}.
\]
Then $\tilde{f}(\tilde{x})$ is a form of even degree $2d$. Consider
the optimization problem: \be \label{maxf:odd->even}
\tilde{f}_{\max} : =\max\limits_{ \tilde{x} \in \mR^{n+1}} \quad
\tilde{f}(\tilde{x}) \quad \text{s.t.} \quad \|\tilde{x} \| = 1. \ee
As shown in \cite[Section~4.2]{Nieform2012}, it holds that
\[
f_{\max} = \sqrt{2d-1}\big(1-\frac{1}{2d}\big)^{-d}
\tilde{f}_{\max}.
\]
Since $\tilde{f}$ is an even form, the semidefinite relaxation for
\reff{maxf:odd->even} is \be  \label{oddf:sdpr}
\tilde{f}_{\max}^{\text{sdp}} := \max\limits_{ y \in \N^{n+1}_{2d} }
\quad \langle \tilde{f}, y \rangle  \quad
 \text{s.t.} \quad M(y) \succeq 0, \quad \langle g, y \rangle = 1.
\ee The vector $y$ is indexed by $(n+1)$-dimensional integer
vectors.

Let $y^*$ be a maximizer of \reff{oddf:sdpr}, which always exists
because the feasible set of \reff{oddf:sdpr} is compact. If
$\rank\,M(y^*) =1$, then \reff{oddf:sdpr} is a tight relaxation, and
a global maximizer $v^+$ of \reff{maxf:odd->even} can be chosen as
in \reff{cho:v+}. (Note that the $n$ in \reff{cho:v+} should be
replaced by $n+1$.) Write $v^+$ as
\[
v^+ = ( \hat{v}, \hat{t} ), \quad \|\hat{v} \|^2 + \hat{t}^2 = 1.
\]
If $\hat{v} =0$ or $\hat{t} = 0$, then $\tilde{f}_{\max}=f_{\max}=0$
and the zero tensor is the best rank-$1$ approximation of $\mF$. So,
we consider the general case $0< |\hat{t}| < 1$. Note that
$\mbox{sign}(\hat{t}) \cdot \hat{v}$ is a global maximizer of $f$ on
the sphere $\|\hat{v}\|_2^2 = 1 - \hat{t}^2$.  Let \be
\label{cho:u:odd} \hat{u} = \mbox{sign}(\hat{t}) \cdot \hat{v} /
\sqrt{ 1 - \hat{t}^2 }. \ee Then $\hat{u}$ is a global maximizer of
$f$ on $\mathbb{S}^{n-1}$. When $\rank\,M(y^*) > 1$, the above
$\hat{u}$ might not be a global maximizer, but it can be used as an
approximation for a maximizer.

Combining the above, we get the following algorithm.

\begin{alg}  \label{alg:sym:odd}
Rank-$1$ approximations for odd symmetric tensors.
\\
\noindent {\bf Input:}\, A symmetric tensor $\mF \in
\mt{S}^m(\re^n)$ with an odd order $m=2d-1$.

\noindent {\bf Output:}\,  A rank-$1$ symmetric tensor $\lmd \cdot
u^{\otimes m}$ with $\lmd \in \re$ and $u \in \mathbb{S}^{n-1}$.

\noindent {\bf Procedure:} \quad \bdes

\item [Step 1] Solve the semidefinite relaxation \reff{oddf:sdpr}
and get an optimizer $y^*$.

\item [Step 2] Choose $v^+$ as in \reff{cho:v+}
and $\hat{u}$ as in \reff{cho:u:odd}. (The $n$ in \reff{cho:v+}
should be replaced by $n+1$.)

\item [Step 3] If $\rank\,M(y^*) = 1$, let $u = \hat{u}$;
otherwise, apply a nonlinear optimization method to get a better
solution $u$ of \reff{maxf:||x||=1}, by using $\hat{u}$ as a
starting point. Let $\lmd = f(u)$. Output $(\lmd, u)$.

\edes

\end{alg}

\noindent {\it Remark:} In Algorithm~\ref{alg:sym:odd}, if
$\rank\,M(y^*) = 1$, then the output $\lmd \cdot u^{\otimes m}$ is a
best rank-$1$ approximation of the tensor $\mF$. If $\rank\,M(y^*) >
1$, then $\lmd \cdot u^{\otimes m}$ is not necessarily the best. The
approximation quality of the semidefinite relaxation
\reff{oddf:sdpr} is analyzed in \cite[Section~4]{Nieform2012}. In
our numerical experiments, we often have $\rank\,M(y^*) = 1$. A
similar version of Theorem~\ref{thm:symeven:rk=1} is true for
Algorithm~\ref{alg:sym:odd}. We omit it for cleanness of the paper.

\subsection{Nonsymmetric tensors}

Let $\mF \in \re^{n_1 \times \cdots \times n_m}$ be a nonsymmetric
tensor of order $m$. Let $F(x^{1}, \ldots, x^{m})$ be the
multi-linear form defined as in \reff{df:F(x):nsy}. To get a best
rank-$1$ approximation of $\mF$ is equivalent to solving the
multilinear optimization problem \reff{maxF:sph:nsym}. Here we show
how to solve it by using semidefinite relaxations.

Without loss of generality, assume $n_m = \max_j n_j$. Since
$F(x^{1}, \ldots, x^{m})$ is linear in $x^m$, we can write it as
\[
F(x^{1}, \ldots, x^{m}) = \sum_{j=1}^{n_m} (x^m)_j F_j(x^{1},
\ldots, x^{m-1}),
\]
where each $F_j(x^{1}, \ldots, x^{m-1})$ is also a multi-linear
form. Let $m^{\prm} = m-1$, and
\[
F^{sq} := \sum_{j=1}^{n_m}  F_j(x^{1}, \ldots, x^{m^\prm})^2.
\]
By the Cauchy-Schwartz inequality, it holds that
\[
|F(x^{1}, \ldots, x^{m})| \leq F^{sq}(x^{1}, \ldots,
x^{m^\prm})^{1/2} \| x^m \| .
\]
The equality occurs in the above if and only if $x^m$ is
proportional to
\[
\big(F_1(x^{1}, \ldots, x^{m^\prm}), \ldots, F_{n_m}(x^{1}, \ldots,
x^{m^\prm}) \big).
\]
Therefore, \reff{maxF:sph:nsym} is equivalent to \be
\label{maxG:nsy:m-1} \left\{ \baray{ccc} F_{\max} :&=
\underset{x^{1}, \ldots, x^{m^\prm} }{\max} &
F^{sq}(x^{1}, \ldots, x^{m^\prm}) \\
& \text{s.t.} &   \|x^{1}\|  = \cdots  =\|x^{m^\prm}\| =1. \earay
\right. \ee

Now we present the semidefinite relaxations for solving
\reff{maxG:nsy:m-1}. The outer product
\[
\mc{K}(x) := x^{1} \otimes \cdots \otimes x^{m^\prm}
\]
is a vector of length $n_1\cdots n_{m^\prm}$.  Denote
\[
\Omega := \{  (\imath, \jmath): \, \imath, \jmath \in [n_1] \times
\cdots \times [n_{m^\prm}] \}.
\]
Expand the outer product of $\mc{K}(x)$ as
\[
\mc{K}(x) \mc{K}(x)^T =   \sum_{ (\imath, \jmath) \in \Omega }
B_{\imath, \jmath }  \, (x^1)_{\imath_1}(x^1)_{\jmath_1} \cdots
(x^{m^\prm})_{\imath_{m^\prm}}(x^{m^\prm})_{\jmath_{m^\prm}},
\]
where each $B_{\imath, \jmath }$ is a constant symmetric matrix. For
$w \in \re^{\Omega}$, define
\[
K(w)  :=   \sum_{ (\imath, \jmath) \in \Omega } B_{\imath, \jmath }
w_{\imath,\jmath}.
\]
Clearly, $K(w)$ is a linear pencil in $w \in \re^{\Omega}$. Write
\[
F^{sq} =   \sum_{ (\imath, \jmath) \in \Omega } G_{\imath, \jmath }
\, (x^1)_{\imath_1}(x^1)_{\jmath_1} \cdots
(x^{m^\prm})_{\imath_{m^\prm}} (x^{m^\prm})_{\jmath_{m^\prm}},
\]
\[
h := \| x^1 \|^2 \cdots \| x^{m^\prm} \|^2 =   \sum_{ (\imath,
\jmath) \in \Omega } h_{\imath, \jmath }  \,
(x^1)_{\imath_1}(x^1)_{\jmath_1} \cdots
(x^{m^\prm})_{\imath_{m^\prm}} (x^{m^\prm})_{\jmath_{m^\prm}}.
\]
For $w \in \re^{\Omega}$, we denote
\[
\langle F^{sq}, w \rangle :=   \sum_{ (\imath, \jmath) \in \Omega }
G_{\imath, \jmath }  \, w_{\imath, \jmath }, \qquad \langle h, w
\rangle :=   \sum_{ (\imath, \jmath) \in \Omega } h_{\imath, \jmath
}  \, w_{\imath, \jmath }.
\]

A semidefinite relaxation of \reff{maxG:nsy:m-1} is \be
\label{max:qG:sdpr} F_{\max}^{\text{sdp}} := \max  \quad \langle
F^{sq}, w \rangle  \quad
 \text{s.t.} \quad K(w) \succeq 0, \quad \langle h, w \rangle = 1.
\ee Define $\Sig_{n_1,\ldots,n_{m^\prm}}$ to be the cone as
\[
\Sig_{n_1,\ldots,n_{m^\prm} } = \left\{ L \left|  \baray{c}
L= L_1^2+\cdots + L_k^2 \mbox{ where each } L_i  \\
 \mbox{is a multilinear form in } (x^1,\ldots, x^{m^\prm})
\earay \right. \right\}.
\]
It can be shown that the dual problem of \reff{max:qG:sdpr} is \be
\min \quad \gamma \quad \text{s.t.} \quad \gamma h - F^{sq} \in
\Sig_{n_1,\ldots,n_{m^\prm} }. \ee Clearly, we always have
$F_{\max}^{\text{sdp}} \geq F_{\max}$. When the equality occurs, we
say that \reff{max:qG:sdpr} is a tight relaxation.

The feasible set of \reff{max:qG:sdpr} is compact, because
\[
\mbox{Trace} (K(w)) = \langle h, w \rangle = 1.
\]
Let $w^*$ be a maximizer of \reff{max:qG:sdpr}. Like for the case of
symmetric tensors, if $\rank\,K(w^*) =1$, then \reff{max:qG:sdpr} is
tight, and there exist vectors $v^1,\ldots,v^{m^\prm}$ of unit
length such that $w^*  = (v^1 (v^1)^T) \otimes \cdots \otimes
(v^{m^\prm}(v^{m^\prm})^T)$ and $(v^1,\ldots,v^{m^\prm})$ is a
maximizer of \reff{maxG:nsy:m-1}. They can be constructed as
follows. Let $\ell \in [n_1] \times \cdots \times [n_{m^\prm}]$ be
the index such that
\[
w^*_{\ell,\ell}  = \max_{ (\imath, \imath) \in \Omega} w^*_{\imath,
\imath}.
\]
Then choose $v^j$ ($j=1,\ldots,m^\prm$) as: \be \label{cho:v1:m-1}
\hat{v}^j = \Big(w^*_{ \hat{\ell}_1,\ell}, w^*_{\hat{\ell}_2,\ell},
\ldots, w^*_{\hat{\ell}_{n_j},\ell}\Big), \qquad v^j = \hat{v}^j /
\| \hat{v}^j \|, \ee where $\hat{\ell}_k= \ell +(k-\ell_j)\cdot e_j$
for each $k\in[n_j]$. When $\rank\,K(w^*)>1$, the tuple
$(v^1,\ldots,v^{m^\prm})$ as in \reff{cho:v1:m-1} might not be a
global maximizer of \reff{maxG:nsy:m-1}. But it can be used as an
approximation for a maximizer of \reff{maxG:nsy:m-1}.

Combining the above, we get the following algorithm.

\begin{alg}  \label{alg:nonsym}
Rank-$1$ approximations for nonsymmetric tensors.
\\
\noindent {\bf Input:}\, A nonsymmetric tensor $\mF \in
\re^{n_1\times   \ldots \times n_m}$.

\noindent {\bf Output:}\,  A rank-$1$ tensor $ \lmd \cdot (u^1
\otimes \cdots \otimes u^m)$ with $\lmd \in \re$ and each $u^j \in
\mathbb{S}^{n_j-1}$.

\noindent {\bf Procedure:} \quad

\bdes

\item [Step 1] Solve the semidefinite relaxation \reff{max:qG:sdpr}
and get a maximizer $w^*$.

\item [Step 2] Choose $(v^1,\ldots,v^{m^\prm})$ as in \reff{cho:v1:m-1}.
Then let
\[
\hat{v}^m := (F_1(v^1,\ldots,v^{m^\prm}), \ldots,
F_{n_m}(v^1,\ldots,v^{m^\prm})),
\]
and $v^m := \hat{v}^m / \|\hat{v}^m\|$.

\item [Step 3] If $\rank \, K(w^*)=1$,
let $u^i = v^i$ for $i=1,\ldots,m$; otherwise, apply a nonlinear
optimization method to get a better solution $(u^1,\ldots,u^m)$ of
\reff{maxF:sph:nsym}, by using $(v^1,\ldots,v^m)$ as a starting
point.

\item [Step 4] Let $\lmd = F(u^1,\ldots, u^m)$, and output
$(\lmd, u^1,\ldots, u^m)$.

\edes

\end{alg}

\noindent {\it Remark:} In Algorithm~\ref{alg:nonsym}, if
$\rank\,K(w^*) = 1$, then the output $\lmd \cdot u^1 \otimes \cdots
\otimes u^m$ is a best rank-$1$ approximation of $\mF$. If
$\rank\,K(w^*) > 1$, then $\lmd \cdot u^1 \otimes \cdots \otimes
u^m$ is not necessarily the best. The approximation quality of the
semidefinite relaxation \reff{max:qG:sdpr} is analyzed in
\cite[Section~3]{Nieform2012}.

We want to know when $\rank \, K(w^*) = 1$. Clearly, for this to be
true, the relaxation \reff{max:qG:sdpr} must be tight, i.e.,
$F_{\max}= F_{\max}^{\text{sdp}}$. Like for the symmetric case, the
reverse is also often true, as shown in the following theorem. Let
$\pt\Sig_{n_1,\ldots,n_{m^\prm}}$ be the boundary of the cone
$\Sig_{n_1,\ldots,n_{m^\prm}}$.

\begin{theorem}
Let $F_{\max},F_{\max}^{\emph{sdp}}, w^*$ be as above. Suppose
$F_{\max} = F_{\max}^{\emph{sdp}}$. If $F_{\max} \cdot h - F^{sq}$
is a smooth point of $\pt\Sig_{n_1,\ldots,n_{m^\prm}}$, then $\rank
\, K(w^*) = 1$.
\end{theorem}
\begin{proof}
This can be proved in the same way as for
Theorem~\ref{thm:symeven:rk=1}. Let $(u^1, \cdots ,u^{m^\prm})$ be a
global maximizer of \reff{maxG:nsy:m-1}. Let $\hat{w} \in
\re^{\Omega}$ be the vector such that
\[
\hat{w}_{\imath, \jmath} = (u^1)_{\imath_1}(u^1)_{\jmath_1} \cdots
(u^m)_{\imath_m}(u^m)_{\jmath_m} \quad \forall \, (\imath,\jmath)
\in \Omega.
\]
The key point is the observation that $\langle p, \hat{w} \rangle =
0$ defines a unique supporting hyperplane of $\Sig_{n_1,\ldots,
n_{m^\prm}}$ through $F_{\max} \cdot h - F^{sq}$, when it is a
smooth point of the boundary $\pt\Sig_{n_1,\ldots, n_{m^\prm}}$. The
proof proceeds same as for Theorem~\ref{thm:symeven:rk=1}.
\end{proof}

\section{Numerical Experiments}\label{sec:exmples}
\setcounter{equation}{0}

In this section, we present numerical experiments of using
semidefinite relaxations to find best rank-$1$ tensor
approximations. The computations are implemented in Matlab 7.10 on a
Dell Linux Desktop with 8GB memory and Intel(R) CPU 2.8GHz. In
applications, the resulting semidefinite programs are often large
scale. Interior point methods are not very suitable for solving such
big semidefinite programs. We use the software {\tt SDPNAL}
\cite{sdpnal} by Zhao, Sun and Toh, which is based on the Newton-CG
augmented Lagrangian method \cite{ZST10}. In our computations, the
default values of the parameters in {\tt SDPNAL} are used. In
Algorithms~\ref{alg:sym:even}, \ref{alg:sym:odd} and
\ref{alg:nonsym}, if the matrices $M(y^*),M(z^*),K(w^*)$ do not have
rank one, we apply the nonlinear program solver {\tt fmincon} in
Matlab Optimization Toolbox to improve the rank-$1$ approximations
obtained from semidefinite relaxations. Our numerical experiments
show that these algorithms are often able to get best rank-$1$
approximations and {\tt SDPNAL} is efficient in solving such large
scale semidefinite relaxations.

We report the consumed computer time in the format {\tt hr:mn:sc}
with {\tt hr} (resp., {\tt mn, sc}) standing for the consumed hours
(resp., minutes, seconds). In our presented tables, {\tt min}
(resp., {\tt med, \tt max}) stands for the minimum (resp., median,
maximum) of quantities like time, errors.

In our computations, the rank of a matrix $A$ is numerically
measured as follows: if the singular values of $A$ are $\sig_1 \geq
\sig_2 \geq \cdots \geq \sig_t>0$, then $\rank\,(A)$ is set to be
the smallest $r$ such that $\sig_{r+1}/\sig_r < 10^{-6}$. In the
display of our computational results, only four decimal digits are
shown.

\subsection{Symmetric tensor examples}

In this subsection, we report numerical experiments for symmetric
tensors. We apply Algorithm~\ref{alg:sym:even} for even symmetric
tensors, and apply Algorithm~\ref{alg:sym:odd} for odd symmetric
tensors. In Algorithm~\ref{alg:sym:even}, if \be
\label{rank:condition:sym:even} \rank\, M(y^*)= \rank\, M(z^*)=1,
\ee then the output tensor $\lmd\cdot u^{\otimes m}$ is a best
rank-$1$ approximation. If $\rank\, M(y^*)>1$ or $\rank\, M(z^*)>1$,
$\lmd\cdot u^{\otimes m}$ is not guaranteed to be a best rank-$1$
approximation. However, the quantity
\[
f_{\text{ubd}} :=\max\{|f_{\max}^{\text{sdp}}|,
|f_{\min}^{\text{sdp}}|\}
\]
is always an upper bound of $|f(x)|$ on $\mathbb{S}^{n-1}$. No
matter whether \reff{rank:condition:sym:even} holds or not, the
error \be \label{aprox:ratio} \text{\tt aprxerr} := \big| |f(u)| -
f_{\text{ubd}} \big|/ \max\{1, f_{\text{ubd}} \} \ee is a measure of
the approximation quality of $\lmd\cdot u^{\otimes m}$. When
Algorithm~\ref{alg:sym:odd} is applied for odd symmetric tensors,
$f_{\text{ubd}}$ and ${\tt aprxerr}$ are defined similarly. As in
Qi~\cite{Qi2011rank1}, we define the {\it best rank-$1$
approximation ratio} of a tensor $\mF \in \mt{S}^m(\re^n)$ as \be
\label{bst:apx:ro} \rho(\mF) := \max\limits_{ \mc{X} \in
\mt{S}^m(\re^n), \rank\, \mc{X} = 1 } \frac{|\langle
\mF,\mc{X}\rangle|}{\|\mF\|\|\mc{X}\|}. \ee If $\lmd \cdot
u^{\otimes m}$, with $\|u\|=1$ and $\lmd = f(u)$, is a best rank-$1$
approximation of $\mF$, then $ \rho(\mF) =  | \lmd | / \|\mF\|. $
Estimates for $\rho(\mF)$ are given in Qi~\cite{Qi2011rank1}.

\medskip

\begin{exm} (\cite[Example 2]{LMV00b})
Consider the tensor $\mF \in \mt{S}^3(\re^2)$ with entries: \be \nn
\begin{small}
\begin{aligned}\label{data:exm::sym:dim2}
& \mF_{111} =1.5578, \quad \mF_{222} = 1.1226,\quad \mF_{112} =
-2.4443,\quad \mF_{221} = -1.0982.
\end{aligned}
\end{small}
\ee When Algorithm~\ref{alg:sym:odd} is applied, we get the rank-$1$
tensor $\lambda \cdot u^{\otimes 3}$ with
\[
\lambda =  3.1155 ,\quad u = (0.9264,-0.3764).
\]
It takes about 0.2 second. The computed matrix $M(y^*)$ has rank
one. So, we know $\lambda \cdot u^{\otimes 3}$ is a best rank-$1$
approximation. The error {\tt aprxerr} = 7.3e-9, the ratio
$\rho(\mF)=0.6203$, the residual $\| \mF - \lmd \cdot u^{\otimes 3}
\| = 3.9399$, and $\|\mF\| = 5.0228$.

\end{exm}

\begin{exm}\label{exm:sym:2}
(\cite[Example~3.6]{shifted2011}, \cite[Example 4.2]{ZLQ12})
Consider the tensor $\mathcal{F} \in \mt{S}^3(\re^3) $ with entries:
\be \nn
\begin{small}
\begin{aligned}\label{data:exm:2}
& \mF_{111} =-0.1281, \mF_{112} = 0.0516, \mF_{113} =-0.0954,
\mF_{122} = -0.1958,\mF_{123} = -0.1790,\\
&   \mF_{133} = -0.2676, \mF_{222} = 0.3251, \mF_{223} =\ \ 0.2513,
\mF_{233} = \quad 0.1773, \mF_{333} =\ \ 0.0338.
\end{aligned}
\end{small}
\ee When Algorithm~\ref{alg:sym:odd} is applied, we get the rank-$1$
tensor $\lambda \cdot u^{\otimes 3}$ with
\[
\lambda =  0.8730, \quad u = ( -0.3921, 0.7249,0.5664).
\]
It takes about 0.2 second. The computed matrix $M(y^*)$ has rank
one, so $\lambda \cdot u^{\otimes 3}$ is a best rank-$1$
approximation. The error {\tt aprxerr}=1.2e-7, the ratio
$\rho(\mF)=0.8890$, the residual $\| \mF - \lmd \cdot u^{\otimes 3}
\| = 0.4498 $ and $\|\mF\| = 0.9820$.

\end{exm}

\begin{exm} (\cite[Example 2]{Qi2011rank1})
Consider the tensor $\mF \in \mt{S}^3(\re^3)$ with entries: \be \nn
\begin{aligned}\label{data:exm:deg3:1}
& \mF_{111} =0.0517 , \mF_{112} = 0.3579 , \mF_{113} = 0.5298,
\mF_{122} = 0.7544,\mF_{123} = 0.2156 ,\\
&   \mF_{133} = 0.3612 , \mF_{222} = 0.3943, \mF_{223} = 0.0146 ,
\mF_{233} = 0.6718, \mF_{333} = 0.9723.
\end{aligned}
\ee 
When Algorithm~\ref{alg:sym:odd} is applied, we get the rank-$1$
tensor $\lambda \cdot u^{\otimes 3}$ with
\[
\lambda = 2.1110,\quad u = (0.5204, 0.5113, 0.6839).
\]
It takes about 0.2 second. Since the computed matrix $M(y^*)$ has
rank one, $\lambda \cdot u^{\otimes 3}$ is a best rank-$1$
approximation. The error {\tt aprxerr}=6.9e-8, the ratio
$\rho(\mF)=0.8574$, the residual $\| \mF - \lmd \cdot u^{\otimes 3}
\| = 1.2672$, and $\|\mF\| = 2.4621$.
\end{exm}

\begin{exm}
(\cite[Example~3.5]{shifted2011}, \cite[Example 4.1]{ZLQ12})
Consider the tensor $\mathcal{F}\in \mt{S}^{4}(\re^{3})$ with
entries:
\begin{equation*}
\begin{small}
\begin{aligned}\label{data:exm:deg4:sym}
  & \mF_{1111} = 0.2883, ~\mF_{1112} =-0.0031, \mF_{1113} =
0.1973, \mF_{1122} = -0.2485, \mF_{1123} = -0.2939, \\
&  \mF_{1133} = 0.3847,~ \mF_{1222} = \quad 0.2972,  \mF_{1223} = 0.1862, \mF_{1233} = ~~~ 0.0919,  \mF_{1333} = -0.3619, \\
 &\mF_{2222} = 0.1241,~\mF_{2223} = -0.3420, ~\mF_{2233} = 0.2127, ~ \mF_{2333} =
0.2727,  \mF_{3333} = -0.3054.
\end{aligned}
\end{small}
\end{equation*}
Applying Algorithm~\ref{alg:sym:even}, we get $\lmd^+ \cdot
(u^+)^{\otimes m}$ and $\lmd^- \cdot (u^-)^{\otimes m}$ with
\[
\lambda^+ = 0.8893, \quad u^+ = (-0.6672,-0.2470,0.7027),
\]
\[
\lambda^- = -1.0954,\quad u^- =  (-0.5915,0.7467,0.3043).
\]
It takes about 0.3 second. Since $|\lmd^+| < |\lmd^-|$, the output
rank-$1$ tensor is $\lambda \cdot u^{\otimes 4}$ with $ \lmd =
\lmd^-, u = u^-$. The computed matrices $M(y^*), M(z^*)$ both have
rank one, so $\lambda \cdot u^{\otimes 4}$ is a best rank-$1$
approximation. The error ${\tt aprxerr} $=2.8e-7, the ratio
$\rho(\mF)=0.4863$, the residual $\| \mF - \lmd \cdot u^{\otimes 4}
\| = 1.9683$, and $\| \mF\|= 2.2525$.
\end{exm}

\begin{exm} \label{sym:example:degree:cubic}
Consider the tensor $\mathcal{F} \in \mt{S}^3(\re^n)$ with entries:
\[
\begin{small}
(\mF)_{i_1,i_2,i_3} =
\frac{(-1)^{i_1}}{i_1} + \frac{(-1)^{i_2}}{i_2} +
\frac{(-1)^{i_3}}{i_3}.
\end{small}
\]
For $n=5$, we apply Algorithm~\ref{alg:sym:odd} and get the rank-$1$
tensor $\lmd \cdot u^{\otimes 3}$ with
\[
\lambda = 9.9779,  \qquad u = -(0.7313,0.1375,0.4674,0.2365,
0.4146).
\]
It takes about 0.3 second. The error {\tt aprxerr}= 1.4e-7. Since
the computed matrix $M(y^*)$ has rank one, we know $\lmd \cdot
u^{\otimes 3}$ is a best rank-$1$ approximation. The ratio
$\rho(\mF)= 0.8813$, the residual $\| \mF - \lmd  \cdot u^{\otimes
3} \| = 5.3498$ and $\|\mF\| = 11.3216$.

For a range of values of $n$ from $10$ to $55$, we apply
Algorithm~\ref{alg:sym:odd} to find rank-$1$ approximations. All the
computed matrices $M(y^*)$ have rank one, so we get best rank-$1$
approximations for all of them. The computational results are listed
in Table~\ref{sym:example3:results:deg3}. For $n=5,10,15,20,25,30$,
Algorithm~\ref{alg:sym:odd} takes less than one minute; for
$n=35,40,45$, it takes a couple of minutes; for $n=50,55$, it takes
about half an hour. The errors {\tt aprxerr} are all very tiny. The
best rank-$1$ approximation ratios are around three quarters.

\begin{table}\centering
\begin{scriptsize}
\begin{tabular}{|c|c|c|c|c||c|c|c|c|c|c|c|c|c|c|} \hline
$n$ &    time  & $\lambda $ & {\tt aprxerr} & $\rho(\mF_n) $ & $n$ &   time & $\lambda $ & {\tt aprxerr} &  $ \rho(\mF_n) $  \\
\hline  10    &  0:00:02&  17.80 & 3.5e-9 & 0.80 &15   & 0:00:03&
26.48 & 5.9e-8 &0.79
\\ \hline 20   &   0:00:06 & 34.16 & 9.8e-8&  0.77 &25    &
 0:00:11& 42.51 & 1.8e-7&0.77
\\ \hline 30    & 0:00:30 &  50.14 & 2.5e-8 &0.75
 & 35    &  0:01:02 & 58.33 & 5.1e-7 &0.75
\\ \hline 40    & 0:04:05 & 65.93 & 1.8e-8 & 0.74 & 45    & 0:12:07& 74.02 &2.5e-8&
0.74
\\ \hline 50    &  0:32:45 &  81.59  & 7.7e-8& 0.73 & 55    & 0:42:00 &  89.62& 9.2e-8&
0.73
 \\ \hline
\end{tabular}
\end{scriptsize}
 \caption{\small{Computational results for Example \ref{sym:example:degree:cubic}.} } \label{sym:example3:results:deg3}
\end{table}

\end{exm}

\begin{exm}  \label{sym:exm:deg4}
Consider the tensor $\mF \in \mt{S}^4(\re^n)$ given as
\[
\begin{small}
(\mF)_{i_1,\cdots,i_4} =
\arctan \Big( (-1)^{i_1} \frac{  i_1}{n} \Big) +\cdots + \arctan
\Big( (-1)^{i_4} \frac{i_4}{n} \Big).
\end{small}
\]
For $n=5$, applying Algorithm~\ref{alg:sym:even}, we get two
rank-$1$ tensors $\lmd^+ \cdot (u^+)^{\otimes m}$ and $\lmd^- \cdot
(u^-)^{\otimes m}$ with \be \nn
\begin{aligned}
\lmd^+  &= \ \ 13.0779, \quad u^+ =
(0.3174, 0.5881,0.1566,0.7260, 0.0418 ),\\
\lambda^- & = -23.5740,  \quad u^- = (0.4403,0.2382,
0.5602,0.1354,0.6459).
\end{aligned}
\ee It takes about 0.6 second. Since $|\lmd^+| < |\lmd^-|$, the
output rank-$1$ tensor is $\lambda \cdot u^{\otimes 4}$ with $\lmd
=-23.5740$ and $u=u^-$. The error ${\tt aprxerr}$=1.4e-7. The
computed matrices $M(y^*), M(z^*)$ both have rank one, so $\lambda
\cdot u^{\otimes 4}$ is a best rank-$1$ approximation. The ratio
$\rho(\mF)=0.8135$, the residual $\| \mF - \lmd  \cdot u^{\otimes 4}
\| = 16.8501$ and $\|\mF\|=28.9769$.

For a range of values of $n$ from $5$ to $60$, we apply
Algorithm~\ref{alg:sym:even} to find rank-$1$ approximations. The
computational results are listed in
Table~\ref{sym:exm:deg4:results}. All the computed matrices $M(y^*),
M(z^*)$ have rank one, so we get best rank-$1$ approximations for
all of them. For $n=5,10,15,20,25,30$, Algorithm~\ref{alg:sym:even}
takes less than one minute; for $n=35,40,45,50,55$, it takes less
than one hour, and for $n=60$, it takes around one hour.
The errors {\tt aprxerr} are all very tiny. The best rank-$1$
approximation ratios are near two thirds.
%

\begin{table}  \centering
\begin{scriptsize}
\begin{tabular}{|c|c|c|c|c||c|c|c|c|c|c|c|c|c|c|} \hline
$n$ &    time  & $\lambda$  & {\tt aprxerr} &  $\rho(\mF) $ & $n$ &   time & $\lambda$ & {\tt aprxerr} & $ \rho(\mF) $ \\
\hline 5   & 0:00:01 &  2.357e+1 &  1.4e-7  &   0.81 &  10 & 0:00:02
& 7.707e+1
 & 4.8e-9 & 0.73
\\ \hline 15   & 0:00:05 & 1.651e+2  &   2.3e-10  & 0.71  &  20 & 0:00:09 & 2.830e+2  &   2.1e-7 &0.69
\\ \hline 25   & 0:00:18 &  4.353e+2 &   5.5e-8 & 0.68  &  30 & 0:00:22 & 6.175e+2  &2.9e-9
&0.68
\\ \hline 35   & 0:01:37 &   8.342e+2 &  9.5e-10     & 0.67  &  40 & 0:04:41 &   1.081e+3 &   1.3e-7    & 0.67
\\ \hline 45   & 0:09:07 & 1.362e+3  &    3.7e-9  &   0.67 &  50 & 0:23:32 &  1.673e+3 & 1.7e-8
&0.67
\\ \hline 55   & 0:45:48  &  2.018e+3  &  3.1e-8    &   0.67  &  60 & 1:04:26 & 2.393e+3 & 1.0e-7 &0.66
\\ \hline
\end{tabular}
\end{scriptsize}
\caption{\small{Computational results for Example
\ref{sym:exm:deg4}. }} \label{sym:exm:deg4:results}
\end{table}

\end{exm}

\begin{exm}\label{sym:example:degree:5}
Consider the tensor $\mathcal{F} \in \mt{S}^5(\re^n) $ given as
\[
\begin{small}
(\mF)_{i_1,\cdots,i_5} =
 (-1)^{i_1}\ln(i_1) + \cdots + (-1)^{i_5}\ln(i_5).
\end{small}
\]
For $n=5$, we apply Algorithm~\ref{alg:sym:odd} and get the rank-$1$
tensor $\lmd \cdot u^{\otimes 5}$ with
\[
\lambda = 110.0083, \qquad u = -(0.3900,0.2785,0.5668,0.1669,0.6490
).
\]
It takes about 0.5 second. The computed matrix $M(y^*)$ has rank
one, so $\lmd \cdot u^{\otimes 5}$ is a best rank-$1$ approximation.
The error {\tt aprxerr}=3.0e-7. The ratio $\rho(\mF)=0.7709$, the
residual $\| \mF - \lmd \cdot u^{\otimes 5} \| = 90.8818$ and
$\|\mF\| = 142.6931$.

For a range of values of $n$ from $5$ to $20$, we apply
Algorithm~\ref{alg:sym:odd} to find rank-$1$ approximations. The
computational results are shown in
Table~\ref{sym:example:degree:5:results}. All the computed matrices
$M(y^*)$ have rank one, so we get best rank-$1$ approximations for
all of them. The errors {\tt aprxerr} are all tiny. The best
rank-$1$ approximation ratios $\rho(\mF)$ are close to $0.70$.
Algorithm~\ref{alg:sym:odd} takes from a few seconds to a couple of
minutes to get best rank-$1$ approximations.

\begin{table} \centering
\begin{scriptsize}
\begin{tabular}{|c|c|c|c|c||c|c|c|c|c|c|c|c|c|c|} \hline
$n$ &  time  & $\lambda$ & {\tt aprxerr} & $ \rho(\mF) $ & $n$ &  time & $\lambda$ & {\tt aprxerr} & $ \rho(\mF) $  \\
\hline 5   &  0:00:01 &  1.100e+2 &  1.2e-7 &  0.77   &13 &0:00:17 &
1.790e+3 &  1.2e-6&  0.69
\\ \hline
6    & 0:00:02 & 2.096e+2 &  2.6e-7 & 0.81 & 14    & 0:00:29 &
2.283e+3 & 1.1e-6 &  0.71
\\ \hline 7    & 0:00:03 &  2.975e+2 & 5.5e-8 & 0.74  &15 &  0:01:18 & 2.697e+3&  1.5e-8&
0.69
\\ \hline 8 & 0:00:06 & 4.706e+2 & 6.7e-8 & 0.77 &16     & 0:02:12 & 3.327e+3 &  8.2e-7 &
0.70
\\ \hline 9    &  0:00:10 & 6.192e+2& 1.4e-7 & 0.72 &  17  & 0:02:47 & 3.854e+3 & 4.9e-7& 0.68
\\ \hline 10    & 0:00:22& 8.833e+2& 1.8e-8 &0.74 &18     & 0:07:54 &   4.637e+3 &  1.4e-7 & 0.69
\\ \hline 11    &  0:00:23 &1.107e+3 & 2.4e-7 &0.70&19     & 0:12:46 &  5.289e+3 &7.2e-7 &
0.68
\\ \hline 12    &   0:00:13 & 1.478e+3& 3.8e-7&   0.72 &20     &  0:22:30 &  6.237e+3& 2.6e-7 &
0.69
\\ \hline
\end{tabular}
\end{scriptsize}
\caption{\small{Computational results for Example
\ref{sym:example:degree:5}.} } \label{sym:example:degree:5:results}
\end{table}

\end{exm}

\begin{exm}\label{sym:exm:Motzkin}
Consider the tensor $\mF \in \mt{S}^6(\re^3)$ with
\begin{equation*}
\begin{small}
\begin{aligned}
& \mF_{111111} = 2, \mF_{111122} = 1/3, \mF_{111133} =2/5, \mF_{1 1
2 2 2 2} =1/3
,\mF_{1 1 2 2 3 3} =1/6, \\
& \mF_{1 1 3 3 3 3} = 2/5, \mF_{2 2 2 2 2 2} = 2  , \mF_{2 2 2 2 3
3} = 2/5, \mF_{2 2 3 3 3 3} =2/5
,\mF_{3 3 3 3 3 3} =1 , \\
\end{aligned}
\end{small}
\end{equation*}
and $\mF_{i_1,\ldots,i_6}=0$ if $(i_1,\ldots, i_6)$ is not a
permutation of an index in the above. We can verify that
\[
f(x) = 2\|x\|^6 -M(x),
\]
where $M(x) = x_1^4x_2^2+x_1^2x_2^4+x_3^6-3x_1^2x_2^2x_3^2$ is the
Motzkin polynomial, which is nonnegative everywhere but not SOS (cf.
\cite{sosexm2000}). Since $0\leq M(x)\leq \|x\|^6$,
we can show that $f_{\max}=2$, $f_{\min}=1$. Applying
Algorithm~\ref{alg:sym:even}, we get
\[
f_{\max}^{\text{sdp}} = 2.0046, \quad v^{+} =(0,1,0) , \quad
f(v^{+})=2,
\]
\[
f_{\min}^{\text{sdp}} = 1.0000, \quad  v^{-}= (0,0,1), \quad
f(v^-)=1.
\]
The matrix $M(z^*)$ has rank one, so $\lmd^- = f(v^-)$ and $u^- =
v^-$. The matrix $M(y^*)$ has rank 7, which is bigger than one, so
we apply ${\tt fmincon}$ to improve $v^+$ but get the same point
$u^+=v^+$; let $\lmd^+ = f(u^+)$. Since $|\lmd^-| < |\lmd^+|$, the
output rank-$1$ tensor is $\lambda \cdot u^{\otimes 6}$ with
\[
\lambda  = 2.0000, \qquad u = (0,1,0 ).
 \]
Since $f_{\max}=\lmd = f(u)$, we know $\lambda \cdot u^{\otimes 6}$
is a best rank-$1$ approximation, by Theorem~\ref{thm:LLM00}. The
best rank-$1$ approximation ratio $\rho(\mF)=0.4046$.
%
%

\end{exm}

\begin{exm}
\label{random:exm:sym:1} (random examples) \, We explore the
performance of Algorithms~\ref{alg:sym:even} and \ref{alg:sym:odd}
on finding best rank-$1$ approximations for randomly generated
symmetric tensors. We generate $\mF \in \mt{S}^m(\re^n)$ with each
entry being a random variable obeying Gaussian distribution (by {\tt
randn} in Matlab). For each generated $\mF$, the semidefinite
relaxations \reff{maxf:sdpr}, \reff{minf:sdpr} and \reff{oddf:sdpr}
can be expressed in the standard dual form \be \label{stand:sdp}
%
%
\left\{ \baray{rl}
\max & b_1 \mu_1 + \cdots + b_M \mu_M \\
\text{s.t.} & F_0 - \sum_{i=1}^M \mu_i F_i \succeq 0, \earay \right.
\ee where $F_i$ are constant symmetric matrices (cf.~\cite{Todd}).
In \reff{stand:sdp}, let $N$ be the length of matrices $F_i$, and
$M$ be the number of variables. For pairs $(n,m)$, if the
semidefinite relaxation matrix length $N<1000$, we test $50$
instances of $\mF$ randomly; otherwise if $N>1000$, we test 10
instances of $\mF$ randomly. For a range of values of $(n,m)$, the
computational results are shown in
Table~\ref{random:best:rank1:sym:all}.

\begin{table} \centering
{ \scriptsize
\begin{tabular}[htb]{|r||c||c|c|c|c|c|c|c|c|} \hline
$(n,m)$ & (N,M)   & \multicolumn{3}{c|}{time (min,med,max)} & {\tt aprxerr} (min,med,max)\\
\hline
(10,3) &  (66,1000)   &  0:00:01& 0:00:01 & 0:00:03& (7.9e-9,~4.5e-8,~2.9e-6)  \\
\hline (20,3) &   (231,10625)     &   0:00:03& 0:00:08& 0:00:13&
(2.4e-9,~3.6e-7,~4.3e-6)
\\ \hline
(30,3) &  (496,46375)     & 0:01:14& 0:01:29& 0:02:01&
(9.1e-9,~7.4e-7,~1.4e-5)
\\ \hline (40,3) &  (861,135750) &   0:06:32& 0:10:04&   0:13:09& (1.3e-9,~4.6e-6,~2.3e-3)
    \\ \hline
    (50,3) & (1326,316250)   & 0:12:39& 0:13:34&0:14:01& (3.2e-9,~1.3e-6,~2.0e-3)
    \\ \hline
(15,4) &  (120,3060)   &   0:00:01& 0:00:03&  0:00:04&
(4.0e-9,~1.1e-7,~1.3e-6)
\\ \hline(20,4) &  (210,8854)  &0:00:52&0:01:09& 0:01:25&(1.2e-8,~1.8e-7,~6.3e-3) \\ \hline
(25,4) & (325,20475)   & 0:00:30 & 0:00:35& 0:00:56& (4.7e-9,~1.3e-7,~1.0e-5)  \\
  \hline(30,4) & (465,40919)   &   0:06:03& 0:07:36&0:09:31&(1.2e-8,~1.1e-6,~9.6e-4) \\ \hline
 (35,4) &   (630,73815)    & 0:02:46& 0:04:57 & 0:06:54&
(4.1e-8,~1.6e-7,~7.4e-3)
 \\
   \hline
 (10,5) &  (286,8007)
 & 0:00:08& 0:00:14& 0:00:17& (4.3e-8,~4.1e-7,~4.1e-6) \\
\hline (15,5) &  (816,54263)    & 0:03:46& 0:03:58 &  0:07:24&
(4.4e-8,~2.5e-6,~1.1e-3)
\\ \hline (20,5)&
 (1771,230229)    & 0:28:14& 0:30:30&  0:43:27& (4.7e-7,~3.7e-6,~5.7e-6) \\ \hline
 (10,6) &  (220,5004)    & 0:00:11 & 0:00:14 & 0:00:20 &  (1.3e-7,~6.4e-7,~3.5e-2) \\
 \hline (15,6) & (680,38759)    &0:03:14 & 0:04:19 &   0:04:53 & (4.8e-8,~2.5e-3,~4.9e-2)
\\ \hline (20,6)
& (1540,177099)   & 0:39:28& 0:45:39 & 0:54:59  & (2.8e-8,~6.6e-5,~1.0e-2)\\
\hline
\end{tabular}
} \caption{\small{Computational results in
Example~\ref{random:exm:sym:1}. Here, $m$ is the tensor order, $n$
is the tensor dimension, $N$ is the length of matrices and $M$ is
the number of variables in the semidefinite relaxations.} }
\label{random:best:rank1:sym:all}

\end{table}

From~Table \ref{random:best:rank1:sym:all}, we can observe that
Algorithms~\ref{alg:sym:even} and \ref{alg:sym:odd} generally
produce accurate best rank-$1$ approximations in a short time. For
some very big problems, like 3-tensors of dimension $40$, or
$4$-tensors of dimension $35$, we are able to get accurate best
rank-$1$ approximations within a reasonable time. For most
instances, we are able to get best rank-$1$ approximations, because
the computed matrices $M(y^*), M(z^*)$ have rank one. For a few
instances, their ranks are bigger than one, and the errors ${\tt
aprxerr}$ are a bit relatively large, like in the order of $10^{-3}$
or $10^{-2}$. This is probably because the semidefinite relaxations
are not very tight.
\end{exm}


\begin{exm}\label{SHOPM:compare}
Here we explore the performance of the Symmetric Higher Order Power
Method (SHOPM) in Kofidis and Regalia~\cite[Algorithm~2]{SHOPM2002},
which is widely used  in getting rank-$1$ approximations for
symmetric tensors. SHOPM can be easily implemented.
Generally, SHOPM generates a sequence of rank-$1$
tensors $\mu_k (v^k)^{\otimes m}$ with each $\|v^k\|=1$ and $\mu_k =
f(v^k)$. It is usually hard to check whether $\mu_k (v^k)^{\otimes
m}$ converges to a best rank-$1$ approximation or not.  Since
semidefinite relaxations often get best rank-$1$ approximations, we
can use them to check convergence of SHOPM. We write a
straightforward Matlab code to implement SHOPM, and use the
truncated HOSVD to generate starting points (cf.~\cite{LMV00b}),
which is commonly used in applications.
We terminate iterations of SHOPM if either $|\mu_{k+1}-\mu_{k}|
< 10^{-8}$ or it runs over $1000$ iterations. The tensor $\mF \in
\mt{S}^m(\re^n)$ we used is given as
\[
\begin{small} \label{SHOPM:compare:deg34}
(\mF)_{i_1,\cdots,i_m} =\sin\big(i_1 + \cdots + i_m \big).
\end{small}
\]
Let $\lmd \cdot u^{\otimes m}$ be the rank-$1$ approximation
generated by Algorithm~\ref{alg:sym:even} (for even orders) or
Algorithm~\ref{alg:sym:odd} (for odd orders), and $\mu \cdot
v^{\otimes m}$ be the rank-$1$ approximation generated by SHOPM.
Their qualities are measured by the residuals:
\[
\text{RES}_{\tt sdp} =\|\mF-\lambda \cdot u^{\otimes m}\|, \qquad
\text{RES}_{\tt shopm} =\|\mF-\mu \cdot v^{\otimes m}\|.
\]
For $m=3,4$ and a range of $n$ (we choose multiples of $5$),
the computational results are presented in Tables
\ref{sym:exm:deg3:SHOPM:results} and
\ref{sym:exm:deg4:SHOPM:results}.
For cleanness of comparisons, only two decimal digits are presented.
For all the cases, the computed matrices $M(y^*)$ and $M(z^*)$ are rank-$1$, so
Algorithms~\ref{alg:sym:even} and \ref{alg:sym:odd} found best
rank-$1$ approximations.
As we can see, for such cases, SHOPM did not get the best, except $(n,m)=(35,3)$.
The computational results for the case $(n,m)=(20,4)$
are not shown, because the software {\tt SDPNAL} experiences
numerical troubles and the semidefinite relaxations cannot be solved accurately.
%
%
As for the consumed time, Algorithms~\ref{alg:sym:even} and
\ref{alg:sym:odd} take up to a few minutes, while SHOPM takes up to
tens of seconds. Generally, SHOPM takes less time, but it
might not find best rank-1 approximations. On the other
hand, the residuals $\text{RES}_{\tt shopm}$, generated by SHOPM,
are not significantly bigger than the residuals $\text{RES}_{\tt
sdp}$, generated by semidefinite relaxations. This shows that SHOPM
performs reasonably well in getting good rank-$1$ approximations,
although it might not get the best.


\begin{table}  \centering
\begin{scriptsize}
\begin{tabular}{|c|r|r|r|r|c|c|c|c|c|c|c|c|c|c|} \hline
$n$   & $|\lambda|$  & $|\mu|$ & $\text{RES}_{\tt sdp}$ & $\text{RES}_{\tt shopm}$ \\
\hline
 10 & 12.12 & 3.01  &   18.79   & 22.16 \\ \hline
     15   & 22.07 &   13.74    &  34.65     &  38.72  \\ \hline
       20  & 32.98 & 26.99 &53.96 & 57.19 \\ \hline
      25   & 44.46  & 33.33  &   76.39   & 81.86  \\ \hline
       30 &  61.23  &   13.03   & 98.75   & 115.46 \\ \hline
     35 &   74.68 &  74.68  &    125.94      &  125.94   \\ \hline
      40 &  92.39     & 53.88    &153.18  & 170.58\\ \hline
\end{tabular}
\end{scriptsize}
\caption{\small{Computational results for Example
\ref{SHOPM:compare:deg34} with $m=3$. }}
\label{sym:exm:deg3:SHOPM:results}
\end{table}

\begin{table}  \centering
\begin{scriptsize}
\begin{tabular}{|c|r|r|r|r|c|c|c|c|c|c|c|c|c|c|} \hline
$n$   & $|\lambda|$  & $|\mu|$ & $\text{RES}_{\tt sdp}$ & $\text{RES}_{\tt shopm}$  \\
\hline  
10 & 27.27 & 1.18  &65.25  & 70.70 \\ \hline
 15   & 61.42   & 50.43   & 146.77   &  150.89 \\ \hline
%
%
 25   &158.22   & 122.02   &   412.65  & 424.76  \\ \hline
  30 &  241.65  &   208.81   &    588.73  & 601.16 \\ \hline
35   &  313.30  &  155.74   &  807.56   &  852.09 \\ \hline
 40 &414.38 & 129.91    &  1052.75 & 1123.89 \\ \hline
\end{tabular}
\end{scriptsize}
\caption{\small{Computational results for Example
\ref{SHOPM:compare:deg34} with $m=4$.}}
\label{sym:exm:deg4:SHOPM:results}
\end{table}

\end{exm}

\subsection{Nonsymmetric tensor examples}

In this subsection, we present numerical results for nonsymmetric
tensors.
In Algorithm~\ref{alg:nonsym}, if $\rank\,K(w^*)=1$, the output
$\lmd\cdot u^{1}\otimes \cdots \otimes u^m$ is a best rank-$1$
approximation of $\mF$. If $\rank\, K(w^*)>1$, $\lmd\cdot
u^{1}\otimes \cdots \otimes u^m$ might not be the best. However, the
quantity
\[
F_{\text{ubd}} := \sqrt{\big|F_{\max}^{\text{sdp}}\big|}
\]
is always an upper bound of $|F(x^1,\ldots,x^m)|$ on
$\mathbb{S}^{n_1-1}\times \cdots \times \mathbb{S}^{n_{m}-1}$. Like
in \reff{aprox:ratio}, we can measure the quality of $\lmd\cdot
u^{1}\otimes \cdots \otimes u^m$ by the error \be
\label{aprox:ratio:nonsym} {\tt aprxerr} := \big|
|F(u^1,\cdots,u^m)|  - F_{\text{ubd}} \big|/ \max\{1, F_{\text{ubd}}
\}. \ee Like the symmetric case, we define the {\it best rank-$1$
approximation ratio} of a tensor $\mF \in \re^{n_1\times \cdots
\times n_m}$ as (cf. Qi~\cite{Qi2011rank1}) \be
\label{bst:apx:ro:nonsym} \rho(\mF) := \max\limits_{ \mc{X} \in
\re^{n_1\times \cdots \times n_m}, \rank \, \mc{X} = 1 }
\frac{|\langle \mF,\mc{X} \rangle| }{ \|\mF\| \| \mc{X} \| }. \ee
Clearly, if $\lmd\cdot u^{1}\otimes \cdots \otimes u^m$, with each
$\|u^i\|=1$ and $\lmd = F(u^{1}, \ldots, u^m)$, is a best rank-$1$
approximation of $\mF$, then $ \rho(\mF) = |\lmd|/\|\mF\|. $

\begin{exm}
(\cite[Example 3]{LMV00b}) Consider the tensor $\mF \in \re^{2
\times 2 \times 2 \times 2}$ with
\begin{equation*}
\begin{small}
\begin{aligned}\label{data:exm:deg4:nonsym}
& \mF_{1111} =25.1, \quad \mF_{1212} = 25.6, \quad \mF_{2121} =
24.8,\quad \mF_{2222} = 23,
\end{aligned}
\end{small}
\end{equation*}
and the resting entries are zeros. Applying
Algorithm~\ref{alg:nonsym}, we get the rank-$1$ tensor $\lmd \cdot
u^1 \otimes u^2 \otimes u^3 \otimes u^4 $ with
\[
\lambda = 25.6000, \quad u^1 = (1,0), \quad u^2 = (0,1),  \quad u^3
= (1,0),  \quad u^4= (0,1).
\]
It takes about 0.3 second. The matrix $K(w^*)$ has rank one, so
$\lmd \cdot u^1 \otimes u^2 \otimes u^3 \otimes u^4 $ is a best
rank-$1$ approximation. The error {\tt aprxerr} =8.9e-10, the ratio
$\rho(\mF)=0.5194$, the residual $\| \mF - \lmd \cdot u^1 \otimes
u^2 \otimes u^3 \otimes u^4 \| = 42.1195$, and $\|\mF\| = 49.2890$.

\end{exm}

\begin{exm}
(\cite[Example 1]{Qi2011rank1}) Consider the tensor $\mF \in \re^{3
\times 3 \times 3}$ with
%
\begin{equation*}
\begin{small}
\begin{aligned}
& \mF_{111} = 0.4333 , \mF_{121} = 0.4278,\mF_{131} = 0.4140, \mF_{211} = 0.8154, \mF_{221}  = 0.0199, \\
& \mF_{231} = 0.5598 , \mF_{311} = 0.0643 ,\mF_{321} =0.3815 , \mF_{331} =0.8834 , \mF_{112} = 0.4866,\\
& \mF_{122} = 0.8087, \mF_{132} = 0.2073, \mF_{212} = 0.7641, \mF_{222} = 0.9924, \mF_{232} = 0.8752, \\
& \mF_{312} = 0.6708, \mF_{322} = 0.8296, \mF_{332} =0.1325, \mF_{113} =0.3871,\mF_{123} =0.0769, \\
& \mF_{133} =0.3151, \mF_{213} =0.1355, \mF_{223} =0.7727, \mF_{233}
=0.4089, \mF_{313} =0.9715,
\\
& \mF_{323} =0.7726, \mF_{333} =0.5526.
\end{aligned}
\end{small}
\end{equation*}
Applying Algorithm~\ref{alg:nonsym}, we get the rank-$1$ tensor
$\lambda \cdot u^1 \otimes u^2 \otimes u^3 $ with
\[
\lambda = 2.8167, \qquad u^1 =  (0.4281,0.6557,0.6220),
\]
\[
u^2 =  (0.5706,0.6467,0.5062), \quad u^3 =  (0.4500,0.7094,0.5424).
\]
It takes less than one second. The matrix $K(w^*)$ has rank one, so
$\lambda \cdot u^1 \otimes u^2 \otimes u^3 $ is a best rank-$1$
approximation. The error {\tt aprxerr} = 3.9e-8, the ratio
$\rho(\mF)=0.9017$, the residual $\| \mF - \lmd \cdot u^1 \otimes
u^2 \otimes u^3 \| = 1.3510$, and $\|\mF\| =3.1239$.

\end{exm}

\begin{exm}(\cite[Section~4.1]{Jocobitensor2013})
Consider the tensor $\mF \in \re^{3 \times 3 \times 3}$ with
\begin{equation*}
\begin{small}
\begin{aligned}
& \mF_{111} = \ \ 0.0072 , \mF_{121} =-0.4413   ,\mF_{131} = \ \ 0.1941 , \mF_{211} = -0.4413 , \mF_{221}  =\ \ 0.0940 , \\
& \mF_{231} = \ \ 0.5901, \mF_{311} = \ \  0.1941  ,\mF_{321} =-0.4099 , \mF_{331} = -0.1012, \mF_{112} = -0.4413,\\
& \mF_{122} = \ \  0.0940 , \mF_{132} =  -0.4099, \mF_{212} = \ \ 0.0940 , \mF_{222} = \ \ 0.2183, \mF_{232} = \ \  0.2950, \\
& \mF_{312} =\ \ 0.5901  , \mF_{322} = \ \ 0.2950 , \mF_{332} = \ \ 0.2229, \mF_{113} =\ \ 0.1941 ,\mF_{123} =\ \ 0.5901, \\
& \mF_{133} =  -0.1012, \mF_{213} = -0.4099, \mF_{223} =\ \  0.2950
, \mF_{233} = \ \  0.2229, \mF_{313} = -0.1012,
\\
& \mF_{323} =\ \  0.2229, \mF_{333} = -0.4891.
\end{aligned}
\end{small}
\end{equation*}
We apply Algorithm~\ref{alg:nonsym}, and get an upper bound
$F_{\text{ubd}} = 1.0000$. The computed matrix $K(w^*)$ has rank
three, so we use {\tt fmincon} to improve the solution and get the
rank-$1$ tensor $\lambda \cdot u^1 \otimes u^2 \otimes u^3 $ with
\[
\lambda = 1.0000, \qquad u^1 =  (0.7955,0.2491,0.5524),
\]
\[
u^2 =  (-0.0050,0.9142,-0.4051), \quad u^3 =
(-0.6060,0.3195,0.7285).
\]
It takes less than one second. Since $\lmd = F(u^1,u^2,u^3) =
F_{\text{ubd}}$, we know $\lambda \cdot u^1 \otimes u^2 \otimes u^3
$ is a best rank-$1$ approximation, by Theorem~\ref{thm:LLM00}. The
error {\tt aprxerr} = 6.0e-9, the ratio $\rho(\mF) = 0.5773$, the
residual $\| \mF - \lmd \cdot u^1 \otimes u^2 \otimes u^3 \| =
1.4143$, and $\|\mF\| =1.7321$.

\end{exm}

\begin{exm}\label{nonsym:deg:3:exm}
Consider the tensor $\mF \in \re^{ n \times n \times n}$ given as
\[
\begin{small}
(\mF)_{i_1,i_2,i_3 } = \cos \big(  i_1+2i_2 +3i_3 \big).
\end{small}
\]
For $n=5$, we apply Algorithm~\ref{alg:nonsym}, and get the rank-$1$
tensor $\lambda \cdot u^1 \otimes u^2 \otimes u^3$ with $\lambda
=6.0996 $ and
\begin{equation*}
\begin{small}
\begin{aligned}
u^1 &= (-0.4296,-0.5611,-0.1767,0.3701,0.5766),\\
u^2 & = (0.6210,-0.2956,-0.3750,0.6077,-0.1308),\\
u^3 & = (-0.4528,0.4590,-0.4561,0.4441, -0.4231).
\end{aligned}
\end{small}
\end{equation*}
It takes around 0.3 seconds. The matrix $K(w^*)$ has rank one, so
$\lambda \cdot u^1 \otimes u^2 \otimes u^3$ is a best rank-$1$
approximation. The error ${\tt aprxerr}$=3.1e-9, the ratio
$\rho(\mF)=0.7728$, the residual $\| \mF - \lmd \cdot u^1 \otimes
u^2 \otimes u^3 \| = 5.0093$ and $\|\mF\|= 7.8930$.

For a range of values of $n$ from $5$ to $40$, we apply
Algorithm~\ref{alg:nonsym} to get rank-$1$ approximations. The
computational results are listed in
Table~\ref{nonsym:deg:3:exm:results}. For all of them, the matrices
$K(w^*)$ have rank one, so we get best rank-$1$ approximations. Most
errors {\tt aprxerr} are very tiny. For $n=40$, the semidefinite
relaxation is very large (the matrices have length $1600$, and there
are $672399$ variables); it was not solved very accurately by {\tt
SDPNAL}.

\begin{table}  \centering
\begin{scriptsize}
\begin{tabular}{|c|c|c|c|c||c|c|c|c|c|} \hline
$n$   &  time & $\lambda$ & {\tt aprxerr}  & $\rho(\mF)$  &  $n$ &   time & $\lambda$ & {\tt aprxerr} & $\rho(\mF)$   \\
\hline  5  &   0:00:01 & 6.100    & 3.1e-9   &  0.77 & 10  &   0:00:02 & 14.79    & 5.8e-9 & 0.66   \\
\hline  15    & 0:00:04 & 25.48  &  2.8e-8  & 0.62 &    20   &   0:00:20 &   33.70  & 1.2e-9 & 0.53   \\
\hline  25 &   0:01:16 & 46.80  & 1.2e-8 &  0.53&  30  &    0:02:57 &   64.91 &   8.2e-9  & 0.56   \\
\hline  35  &    0:13:09  &  80.77   & 6.5e-10  & 0.55&  40  &    0:27:07 &  95.09  &   4.5e-5  &0.53 \\
 \hline
\end{tabular}
\end{scriptsize}
\caption{\small{Computational results for Example
\ref{nonsym:deg:3:exm}.}} \label{nonsym:deg:3:exm:results}
\end{table}

\end{exm}

\begin{exm}
\label{nonsym:exm:deg4} Consider the tensor $\mathcal{F} \in
\re^{n\times n \times n \times n}$ given as
\begin{equation*}
\begin{small}
     (\mF)_{i_1,\cdots,i_4} =\left\{
        \begin{aligned}
        & \sum\limits_{j=1}^4 \arcsin\Big((-1)^{i_j}
          \frac{j}{i_j} \Big),\quad  & \text{if all}\,\, i_j\geq j,\\
        & 0,& \text{otherwise}.
        \end{aligned}
    \right.
    \end{small}
\end{equation*}
For $n=5$, we apply Algorithm~\ref{alg:nonsym}, and get the rank-$1$
tensor $\lambda \cdot u^1 \otimes u^2 \otimes u^3\otimes u^4$ with
$\lambda = 15.3155 $ and \be \nn
\begin{small}
\begin{aligned}
u^1  &= (0.6711,0.2776,0.4398,0.3285,0.4138),\\
u^2  &=  (0,0.1709,0.6708,0.3985,0.6017),\\
u^3  &= (0,0,0.8048,0.1805,0.5655),\\
 u^4 & = (0,0,0,-0.0073,-0.9999).
\end{aligned}
\end{small}
\ee It takes about 3.8 seconds. The matrix $K(w^*)$ has rank one, so
$\lambda \cdot u^1 \otimes u^2 \otimes u^3\otimes u^4$ is a best
rank-$1$ approximation. The error ${\tt aprxerr}$=9.2e-10, the ratio
$\rho(\mF)=0.7076$, the residual $\| \mF - \lmd \cdot u^1 \otimes
u^2 \otimes u^3 \| = 15.2957$, and $\|\mF\|=21.6454$.

For a range of values of $n$ from $5$ to $12$, we apply
Algorithm~\ref{alg:nonsym} to find rank-$1$ approximations. The
computational results are shown in Table
\ref{nonsym:exm:deg4:results}. For all of them, the matrices
$K(w^*)$ have rank 1, so we get best rank-$1$ approximations. The
approximation errors ${\tt aprxerr}$ are all very tiny. For $n\leq
8$, Algorithm~\ref{alg:nonsym} takes a few seconds; for $n = 9, 10,
11$, it takes a couple of minutes. For $n=12$, it takes about four
hours; in this case, the semidefinite relaxation is very big (the
matrices have length $1728$ and there are $474551$ variables).
%
%
\begin{table}
  \centering
\begin{scriptsize}
\begin{tabular}{|c|c|c|c|c||c|c|c|c|c|} \hline
  $n $    &  time & $\lambda$ & {\tt aprxerr} & $\rho(\mF)$ &  $n $    &  time & $\lambda$ & {\tt aprxerr} & $\rho(\mF)$  \\
\hline 5  & 0:00:04 & 15.32 & 9.2e-10& 0.71&
 6  &   0:00:06 &  25.39     &   6.7e-9& 0.75      \\  \hline
 7  & 0:00:17&  33.04&   7.2e-8&   0.69 &
8  &  0:00:30 & 44.66   &   3.8e-8& 0.71   \\
\hline
 9 & 0:02:44&   54.40 &  5.9e-8& 0.68 &
10  &  0:05:56&  67.35    &   1.1e-8& 0.69 \\
\hline
 11 &  0:14:13&   78.93 &  5.9e-9&   0.67 &
 12  & 2:00:50& 93.02 &   2.4e-8&  0.68   \\
 \hline
\end{tabular}
\end{scriptsize}
\caption{\small{Computational results for Example
\ref{nonsym:exm:deg4}.}}\label{nonsym:exm:deg4:results}
\end{table}

\end{exm}

\begin{exm}\label{nonsym:deg:5:exm}
Consider the tensor $\mF \in \re^{n \times n \times n \times n
\times n}$ given as
\[
\begin{small}
(\mF)_{i_1,i_2,i_3,i_4,i_5} = \sum\limits_{j=1}^5(-1)^{j+1}\cdot j
\cdot \exp\{ -i_j \}.
 \end{small}
\]
For $n=4,5,6$, we apply Algorithm~\ref{alg:nonsym} to get rank-$1$
approximations. The computational results are listed in
Table~\ref{nonsym:deg:5:exm:results}. All the computed matrices
$K(w^*)$ have rank one, so best rank-$1$ approximations are found
for all of them. The errors {\tt aprxerr} are all very tiny. In
Table~\ref{nonsym:deg:5:exm:results}, the pair $(N,M)$ measures the
sizes of semidefinite relaxations, with $N$ the length of matrices
and $M$ the number of variables. For $n=4,5$, it takes a short time;
for $n=6$, it takes about 20 minutes.

\begin{table}
\begin{scriptsize}
\centering
\begin{tabular}{ |c|c|c|c|l| }
\hline    $n$& items  & values & $u^i$ & \quad \quad \quad \quad \quad vector entries  \\
\hline
  \multirow{5}{*}{$4$} & (N,M)   & (256,9999)  & $u^1$& $(0.5776,0.4950,0.4646,0.4534)$\\
   & $ \lambda$ &  30.1125   &$u^2$ & $(0.3279,0.4956,0.5573,0.5800)$\\
 &   time  & 0:00:07 &$u^3$ & $(0.7268,0.4679,0.3727,0.3376)$\\
 & {\tt aprxerr}  &  1.1e-10  &$u^4$ & $(0.0998,0.4636,0.5974,0.6467)$\\
 &$\rho(\mF)$   & 0.85    &$u^5$ & $(0.8982,0.3793,0.1884,0.1182)$ \\ \hline
\multirow{5}{*}{5} & (N,M)   &  (625,50624)      & $u^1$& $(0.5282,0.4515,0.4233,0.4129,0.4091)$\\
   & $ \lambda$ & 48.8437    &$u^2$ & $(0.2691,0.4249,0.4822,0.5033,0.5110)$\\
 &   time  & 0:01:20 &$u^3$ & $(0.6865,0.4456,0.3570,0.3244,0.3124)$\\
 & {\tt aprxerr}  & 3.3e-7   &$u^4$ & $(0.0339,0.3734,0.4983,0.5442,0.5612)$\\
 &$\rho(\mF)$   &   0.83   &$u^5$ & $(0.8822,0.3861,0.2037,0.1365,0.1118)$ \\ \hline
 \multirow{5}{*}{6} & (N,M)   &  (1296,194480)     & $u^1$& $(0.4915,0.4183,0.3915,0.3816,0.3779,0.3766)$\\
   & $ \lambda$ & 71.9071  &$u^2$ & $(0.2266,0.3751,0.4298,0.4499,0.4573,0.4600)$\\
 &   time  & 0:19:58  &$u^3$ & $(0.6557,0.4259,0.3413,0.3102,0.2988,0.2946)$\\
 & {\tt aprxerr}  &   8.6e-8  &$u^4$ & $(-0.0120,0.3124, 0.4317,0.4757,0.4918,0.4977)$\\
 &$\rho(\mF)$   & 0.81  &$u^5$ & $(0.8707,0.3875,0.2097,0.1443,0.1202,0.1114)$ \\ \hline
\end{tabular}
\end{scriptsize}
\caption{\small{Computational results for Example
\ref{nonsym:deg:5:exm}.} }\label{nonsym:deg:5:exm:results}
\end{table}

\end{exm}

\begin{exm}
\label{nonsym:exm:fail} Let $B$ be the symmetric matrix {\small
\[
 \bbm 1 & 0 & 0 & 0 & -1/2 & 0& 0& 0& -1/2 \\
0 &2 &0 &-1/2 &0 &0& 0 &0 &0 \\
0 &0 &0& 0& 0 &0 &-1/2 &0 &0 \\
0 &-1/2& 0 &0 &0& 0& 0 &0 &0 \\
-1/2& 0& 0 &0 &1 &0& 0 &0 &-1/2\\
0 &0& 0& 0& 0& 2& 0& -1/2& 0 \\
0 &0 &-1/2 &0 &0 &0 &2 &0 &0 \\
0 &0 &0 &0& 0& -1/2& 0& 0& 0 \\
-1/2 &0 &0 &0 &-1/2& 0& 0& 0 &1 \ebm .
\]
}
The eigenvalues of $B$ are
\[
\frac{2-\sqrt{5}}{2}, \quad 0,  \quad \frac{3}{2}, \quad
\frac{2+\sqrt{5}}{2},
\]
which are all less than $3$. Consider the tensor $\mF \in \re^{3
\times 3 \times 9}$ such that
\[
F^{sq}(x^1,x^2) = (x^1 \otimes x^2)^T (3 I_9 - B ) (x^1 \otimes
x^2).
\]
The bi-quadratic form $3\|x^1\|_2^2 \|x^2\|_2^2 - F^{sq}(x^1,x^2)$
is nonnegative but not SOS (cf.~\cite{biQudnotsos1975}). The minimum
of $(x^1 \otimes x^2)^T B (x^1 \otimes x^2)$ over $\mathbb{S}^2
\times \mathbb{S}^2$ is zero (\cite[Example 5.1]{NieBQ2009}), so
$F_{\max} = 3$. We apply Algorithm~\ref{alg:nonsym}. The computed
matrix $K(w^*)$ has rank 4, which is bigger than one. The upper
bound $F_{\max}^{\text{sdp}} = 3.0972$. Applying {\tt fmincon}, we
get the improved tuple $(u^1,u^2,u^3)$ and $\lmd$ as
\[
u^1 = (0,1,0), ~ u^2 = (1,0,0),~ u^3 =
(0,0.1246,0,0,0,0,-0.9922,0,0),
\]
\[
 \lambda  = F(u^1,u^2,u^3)=1.7321 = \sqrt{F_{\rm max}}.
\]
So, $\lambda \cdot u^1 \otimes u^2 \otimes u^3$ is a best rank-$1$
approximation, by Theorem~\ref{thm:LLM00}. The ratio
$\rho(\mF)=0.4083$, the residual $\| \mF - \lmd \cdot u^1 \otimes
u^2 \otimes u^3 \|= 3.8730$ and $\|\mF \| =4.2426$.

\end{exm}

\begin{exm} \label{random:exm:nonsym:1}
(random examples) 
We explore the performance of
Algorithm~\ref{alg:nonsym} on randomly generated nonsymmetric
tensors $\mF \in \re^{n_1 \times \cdots \times n_m}$. The entries of
$\mF$ are generated obeying Gaussian distributions (by {\tt randn}
in Matlab). As in \reff{stand:sdp}, let $N$ be the length of
matrices and $M$ be the number of variables in the semidefinite
relaxations. If $N<1000$, we generate 50 instances of $\mF$
randomly; if $N>1000$, we generate 10 instances of $\mF$ randomly.
We apply Algorithm~\ref{alg:nonsym} to get rank-$1$ approximations.
The computational results for orders $m=3,4,5$ are shown in Tables
\ref{random:best:rank1:cubic:tensor},
\ref{random:best:rank1:4order:tensor}, and
\ref{random:best:rank1:5order:tensor} respectively. When $n_1
=\cdots = n_m$, the sizes of the semidefinite relaxations are
typically much larger than the sizes for symmetric tensors. For
instance, for $(n,m)=(40,3)$, $(N,M)=(861,135750)$ for the symmetric
case, while $(N,M)=(1600,672399)$ for the nonsymmetric case.
Typically, Algorithm~\ref{alg:nonsym} takes more time than
Algorithms~\ref{alg:sym:even} and \ref{alg:sym:odd}, when the input
tensors have same dimensions and orders.

We can observe from Tables \ref{random:best:rank1:cubic:tensor},
\ref{random:best:rank1:4order:tensor}, and
\ref{random:best:rank1:5order:tensor} that for most instances,
Algorithm~\ref{alg:nonsym} is able to get best rank-$1$
approximations very accurately, within a reasonable short time. For
a few cases, the errors are a bit relatively large, around
$10^{-2}$, which is probably because the semidefinite relaxations
are not very tight.

\begin{table}
 \centering
\begin{scriptsize}
\begin{tabular}{|c| c| c|c|c|c|c|c|} \hline
$(n_1,n_2,n_3)$ &   \multicolumn{3}{c|}{time (min,med,max)}&  {\tt aprxerr} (min,med,max) \\
\hline $(10\times 10\times 10)$ & 0:00:02 & 0:00:02& 0:00:03& (1.6e-9,2.2e-8,2.7e-7) \\
\hline $(15\times 15 \times 15)$  & 0:00:10 &0:00:12&0:00:18&(5.9e-9,5.8e-7,2.1e-3)\\
\hline $(20\times 20 \times 20)$  & 0:00:05& 0:00:48&0:01:24&(1.9e-9,5.7e-7,5.2e-3)\\
\hline  $(25\times 25 \times 25)$  &  0:00:40& 0:02:26&  0:04:57&
(3.2e-9,4.6e-7,5.1e-2)
\\
\hline  $(30\times 30 \times 30)$   &   0:05:48& 0:07:57& 0:11:31&(3.6e-8,1.6e-3,3.5e-2)\\
\hline $(35\times 35\times 35)$   & 0:21:43& 0:27:04& 1:00:05&(1.1e-5,7.7e-3,5.7e-2)\\
\hline  $(40\times 40 \times 40)$ &  1:10:05&1:30:24&1:36:24&(4.8e-4,9.4e-4,1.4e-2)\\
\hline
\end{tabular}
\end{scriptsize}
\caption{\small{Computational results for Example
\ref{random:exm:nonsym:1} with
$m=3$.}}\label{random:best:rank1:cubic:tensor}
\end{table}

\begin{table}  \centering
\begin{scriptsize}
\begin{tabular}{|c|c|c|c|c|c|c|c|} \hline
 ($n_1,n_2,n_3,n_4$)  &  \multicolumn{3}{c|}{time (min,med,max)}&  {\tt aprxerr} (min,med,max) \\
\hline ($\ 5\times\ 5\times \ 5 \times \ 5)$ &  0:00:02&0:00:03&0:00:05&(1.0e-10,1.7e-8,3.1e-7)   \\
\hline ($\ 8\times\ 8\times \ 8 \times \ 8)$ &  0:00:09& 0:00:17 &0:00:28& (2.3e-7,1.5e-6,1.1e-5)   \\
\hline
 ($10\times10\times 10 \times 10)$ & 0:00:57& 0:01:52& 0:10:24&(9.4e-8,2.0e-6,6.6e-3)  \\  \hline
  ($15\times15\times \ 5 \times 15)$ & 0:02:18& 0:07:53& 0:13:54&(1.8e-7,3.7e-7,3.7e-3) \\  \hline
 ($12\times12\times 12 \times 12)$ & 0:07:24&0:10:47&0:39:07&(1.7e-7,3.3e-6,2.7e-2)   \\  \hline
 ($20\times20\times \ 5 \times 20)$ & 1:23:06 &1:55:14& 2:36:27&(4.0e-8,3.7e-4,2.0e-2) \\\hline
\end{tabular}
\end{scriptsize}
\caption{\small{Computational results for Example
\ref{random:exm:nonsym:1} with $m=4$.}}
\label{random:best:rank1:4order:tensor}
\end{table}

\begin{table}\centering
\begin{scriptsize}
\begin{tabular}{|c|c|c|c|c|c|c|c|} \hline
 ($n_1,n_2,n_3,n_4,n_5$)  &  \multicolumn{3}{c|}{time (min,med,max)}& {\tt aprxerr} (min,med,max)  \\
\hline ($  5\times 5\times 5 \times 5 \times 5)$ & 0:00:14&0:00:24&0:00:35& (7.2e-8,3.7e-7,3.6e-6)  \\
\hline ($10\times5\times 5 \times 4 \times 10)$ &
0:00:57&0:01:20&0:03:27&(9.7e-8,4.7e-7,1.5e-5)  \\ \hline
($10\times5\times 8 \times 5\times 10)$ & 0:21:16& 0:40:06& 1:22:28& (2.1e-6,1.3e-4,2.2e-3)   \\
\hline
 ($8\times8\times 8 \times 4\times 10)$ &  1:11:02&1:24:29& 2:53:40&(1.4e-7,2.4e-3,1.6e-2) \\
\hline
\end{tabular}
\end{scriptsize}
\caption{\small{Computational results for Example
\ref{random:exm:nonsym:1} with $m=5$.}}
\label{random:best:rank1:5order:tensor}
\end{table}
\end{exm}


\begin{exm}\label{HOPM:compare:deg34}
In this example, we explore the performance of the Higher Order
Power Method (HOPM) proposed in De Lathauwer~et
al.~\cite[Algorithm~3.2]{LMV00b}, which is widely used for computing
rank-$1$ approximations for nonsymmetric tensors. HOPM can be easily
implemented. Typically, HOPM generates a sequence of rank-$1$
tensors $\mu_k v^{k,1}\otimes \cdots  \otimes v^{k,m}$ (with each
$\|v^{k,j}\|=1$). It is usually very hard to check whether this
sequence converges to a best rank-$1$ approximation or not. Since
Algorithm~\ref{alg:nonsym} often produces best rank-$1$
approximations, we can use it to check convergence of HOPM. We write
a straightforward Matlab code to implement HOPM, and terminate its
iterations if either $|\mu_{k+1}-\mu_k| < 10^{-8}$ or it runs over
$1000$ iterations. Like for the symmetric case, we use the truncated
HOSVD (cf.~\cite{LMV00b}) to generate starting points. The tensor $\mF \in
\re^{n \times \cdots \times n}$ we used is given as
\[
\begin{small}
(\mF)_{i_1,\cdots,i_m} =
  \tan \left(i_1 - \frac{i_2 }{2} + \cdots + (-1)^{m+1} \frac{i_m}{m}  \right).
 \end{small}
\]
Let $\lmd \cdot u^1 \otimes \cdots \otimes  u^m$ be the rank-$1$
approximation returned by Algorithm~\ref{alg:nonsym}, and $\mu \cdot
v^1 \otimes \cdots \otimes v^m$ be the one returned by HOPM.
As in Example \ref{SHOPM:compare}, we measure their qualities by the
residuals:
\[
\text{RES}_{\tt sdp} =\|\mF-\lambda \cdot u^1 \otimes \cdots
 \otimes u^m \|, \qquad \text{RES}_{\tt hopm} =\|\mF-\mu
\cdot v^1 \otimes \cdots \otimes  v^m \|.
\]
For $m=3,4$ and a range of $n$, the computational results are presented in Tables
\ref{nonsym:exm:deg3:HOPM:results} and \ref{nonsym:exm:deg4:HOPM:results}.
For cleanness of comparisons, only two decimal digits are presented.
For all the cases, Algorithm~\ref{alg:nonsym} produced best rank-$1$
approximations (because the computed matrices $K(w^*)$ have rank one).
HOPM does not get the best for most of them.
The computational results for the case $(n,m)=(10,4)$
are not show, because the software {\tt SDPNAL}
experiences numerical troubles and the semidefinite relaxations
cannot be solved accurately.
%
%
%
As for the consumed time, Algorithm~\ref{alg:nonsym} takes up to a
few minutes, while HOPM takes up to tens of seconds. Generally,
HOPM takes less time, but it might not find best rank-1
approximations. On the other hand, the residual $\text{RES}_{\tt
hopm}$, generated by HOPM, is not significantly bigger than
$\text{RES}_{\tt sdp}$, generated by semidefinite relaxations. This
shows that HOPM still gives reasonably well rank-1 approximations,
though it might not produce the best.


\begin{table}  \centering
\begin{scriptsize}
\begin{tabular}{|c|c|c|r|r|c|c|c|c|c|c|c|c|c|c|} \hline
$n$   & $|\lambda|$  & $|\mu|$ & $\text{RES}_{\tt sdp}$   & $\text{RES}_{\tt hopm}$ \\
\hline  15   &449.19    &  349.73  &  1264.41   &   1295.45     \\
\hline 20 & 508.82   & 382.66  &   1780.56   &1811.88   \\
\hline 25   & 579.62   & 442.98   &  2304.82  &   2334.94    \\
\hline 30 &   655.10   &  492.50 &  3022.60  & 3053.31  \\
\hline 35 &  709.39  &   555.07 & 3595.36   &3622.40   \\
\hline
\end{tabular}
\end{scriptsize}
\caption{\small{Computational results for Example
\ref{HOPM:compare:deg34} with $m=3$.} }
\label{nonsym:exm:deg3:HOPM:results}
\end{table}

\begin{table}  \centering
\begin{scriptsize}
\begin{tabular}{|c|c|c|r|r|c|c|c|c|c|c|c|c|c|c|} \hline
$n$   & $|\lambda|$  & $|\mu|$ & $\text{RES}_{\tt sdp}$ & $\text{RES}_{\tt hopm}$  \\
%
%
\hline 7   &  180.22   & 171.39   &  596.18 &   598.78      \\
\hline 8 & 191.73   & 117.53  & 691.98    & 708.37  \\
\hline 9   & 233.49  &  211.69   &   904.78  & 910.13       \\
%
%
\hline 11   & 386.04 &  379.96 &   1512.92   &  1514.71    \\
\hline 12 &  550.21  & 357.85 & 2069.34   & 2111.13
\\ \hline
\end{tabular}
\end{scriptsize}
\caption{\small{Computational results for Example
\ref{HOPM:compare:deg34} with $m=4$.}}
\label{nonsym:exm:deg4:HOPM:results}
\end{table}

\end{exm}

\section{Conclusions and discussions}

This paper proposes semidefinite relaxations to find best rank-$1$
approximations, for both symmetric and nonsymmetric tensors. Three
algorithms based on semidefinite relaxations are presented,
respectively for even symmetric tensors, odd symmetric tensors and
nonsymmetric tensors. As shown in our numerical experiments, they
very often produce best rank-$1$ approximations, which can be
checked mathematically.

In Section~\ref{section:sdp:relaxation}, we only presented the
lowest order semidefinite relaxations. When they are not tight,
higher order semidefinite relaxations can be applied, and we can get
a convergent hierarchy of semidefinite relaxations, as shown by
Lasserre \cite{Las01}. Indeed, this hierarchy almost always
converges within finitely many steps, as recently shown in
\cite{Nie-opcd}. A question that is closely related to best rank-1
approximations is to compute extreme Z-eigenvalues for symmetric
tensors of even orders. Hu, Huang and Qi~\cite{Hu2013} proposed a
convergent sequence of SOS relaxations for computing maximum or
minimum Z-eigenvalues.

Semidefinite relaxations in rank-$1$ tensor approximations are often
large scale. The traditional interior point methods are generally
too expensive to be used for solving such SDPs. In our work, the
Newton-CG augmented Lagrange method by Zhao, Sun and Toh
\cite{ZST10} is applied. The software {\tt SDPNAL} \cite{sdpnal} is
based on this method. It is very suitable for solving such large
scale SDPs. In most of our numerical experiments, {\tt SDPNAL}
successfully solved the semidefinite relaxations, and  we got best
rank-$1$ approximations. For a few cases, {\tt SDPNAL} has trouble
to get accurate solutions. This is probably because these
semidefinite programs are degenerate. Typically, {\tt SDPNAL}
\cite{sdpnal} works well when the SDP is nondegenerate. For
degenerate SDPs, {\tt SDPNAL} might experience numerical troubles
and may have very slow convergence. It is an important future work
to design efficient methods for solving large scale, possibly
degenerate, semidefinite programs arising from tensor
approximations.

In practice, the optimal matrices of semidefinite relaxations
\reff{maxf:sdpr}, \reff{minf:sdpr}, and \reff{max:qG:sdpr} often have rank-1. For
such semidefinite relaxations, the low rank methods by Burer and
Monteiro \cite{Burer2003,Burer2005} can be applied.
They can also be very efficient in applications. On the other
hand, these kinds of methods are not always guaranteed to get
optimal solutions of semidefinite relaxations, because local nonlinear
optimization methods are mainly used.
%
%

In Algorithms~\ref{alg:sym:even}, \ref{alg:sym:odd} and
\ref{alg:nonsym}, if the matrices $M(y^*)$, $M(z^*)$, $K(w^*)$ have
rank one, we are guaranteed to get best rank-$1$ approximations. If
their ranks are bigger than one, we can improve the rank-$1$
approximation by using some nonlinear optimization methods. In our
numerical experiments, we used the Matlab Optimization Toolbox
function {\tt fmincon} to do this. More advanced nonlinear
optimization methods can also be applied, e.g., the Quasi-Newton
method by Savas and Lim \cite{SavLim10}.
%
%

\bigskip
\noindent {\bf Acknowledgement} \, The research was partially
supported by the NSF grant DMS-0844775.
The authors would like very much
to thank the associate editor and two anonymous referees
for fruitful comments on improving the paper.


\end{document}